%% file: UncertaintyCalibration.tex
\RequirePackage{fix-cm}
\documentclass[twocolumn, nospthms]{svjour3}
\usepackage[utf8]{inputenc}
\usepackage{natbib}
\usepackage{amsmath}
\usepackage{amssymb}
\usepackage{tikz}
\usetikzlibrary{arrows.meta}
\usepackage{hyperref}
\usepackage{graphicx}
\usepackage{algorithmic}
\usepackage{algorithm}
\usepackage{subfig}
\usepackage{authblk}
\usepackage[toc,page]{appendix}
\usepackage{amsthm, thmtools}
\usepackage{thm-restate}
\usepackage{float}
\usetikzlibrary{positioning}

\newtheorem{theorem}{Theorem}

\newtheorem{lemma}{Lemma}
\newtheorem{conjecture}{Conjecture}

\theoremstyle{remark}

\title{Uncertainty calibration for probabilistic projection methods}
\author{Vladimir Fanaskov}
\institute{Center for Design, Manufacturing, and Materials, Skoltech\\\email{Vladimir.Fanaskov@skoltech.ru}}
\titlerunning{Uncertainty calibration for probabilistic projection methods}
\date{\today}

\begin{document}
\maketitle
\begin{abstract}
    Classical Krylov subspace projection methods for the solution of linear problem $Ax = b$ output an approximate solution $\widetilde{x}\simeq x$. Recently, it has been recognized that projection methods can be understood from a statistical perspective. These probabilistic projection methods return a distribution $p(\widetilde{x})$ in place of a point estimate $\widetilde{x}$. The resulting uncertainty, codified as a distribution, can, in theory, be meaningfully combined with other uncertainties, can be propagated through computational pipelines, and can be used in the framework of probabilistic decision theory. The problem we address is that the current probabilistic projection methods lead to the poorly calibrated posterior distribution. We improve the covariance matrix from previous works in a way that it does not contain such undesirable objects as  $A^{-1}$ or $A^{-1}A^{-T}$, results in nontrivial uncertainty,  and reproduces an arbitrary projection method as a mean of the posterior distribution. We also propose a variant that is numerically inexpensive in the case the uncertainty is calibrated a priori. Since it usually is not, we put forward a practical way to calibrate uncertainty that performs reasonably well, albeit at the expense of roughly doubling the numerical cost of the underlying projection method.
\end{abstract}
\section{Introduction}
\label{section:Introduction}
One way to approximately solve $Ax = b, A\in\mathbb{R}^{n\times n}$ is to start from the initial guess $x_0$, choose two subspaces $\mathcal{K}, \mathcal{L}$ spanned by columns of matrices $V, W\in\mathbb{R}^{n\times m}$, $m\leq n$ and enforce Petrov--Galerkin condition: $\widetilde{x} = x_0 + \delta,~\delta\in\mathcal{K},~b - A\widetilde{x}\perp \mathcal{L}$. For suitably chosen subspaces, the new approximation reads
\begin{equation}\label{projection_method}
    \widetilde{x} = x_0 + V \left(W^T A V\right)^{-1}W^{T}\left(b - A x_{0}\right).
\end{equation}
Different choices of $V, W$ lead to different projection methods, amongst which are conjugate gradient algorithm, generalized minimum residual method, and others \cite{saad2003iterative}.

A series of  papers starting with the work on probabilistic reconstruction of quasi-Newton methods \cite{hennig2013quasi} led to Bayesian projection methods \cite{hennig2015probabilistic}, \cite{cockayne2018bayesian}, \cite{bartels2019probabilistic}. In contrast to classical projection methods that provide point estimation \eqref{projection_method}, probabilistic projection methods produces a distribution $p(\widetilde{x})$ that reflects uncertainty about the true solution $A^{-1}b$. In particular, in \cite{cockayne2018bayesian} and \cite{bartels2019probabilistic}, the  authors proved the following result:
\begin{theorem}
\label{theorem:prior_art}
    Let $\det A \neq 0$,  $p(x) = \mathcal{N}(x|x_0, \Sigma_{0})$ and $y_{m} = S_{m}^{T}Ax$, where $S_{m}\in\mathbb{R}^{n\times m}, m\leq n$ is a full-rank matrix. The mean of conditional distribution $p(x|y_{m} = S^{T}_{m}b) = \mathcal{N}(x|x_m, \Sigma_{m})$ reproduces projection method \eqref{projection_method} for three choices of prior distribution and search directions $S_{m}$:
    \begin{enumerate}
        \item $\Sigma_{0} = V V^{T}$ and $S_{m} = W$ result in $x_{m} = \widetilde{x}$, $\Sigma_{m} = 0$;
        \item In case $A$ is symmetric positive definite, the choice $\Sigma_{0} = A^{-1}$, $S_{m} = V$ results in $x_{m} = \left.\widetilde{x}\right|_{W=V}$, $\Sigma_{m} = A^{-1} - V \left(V^{T} A V\right)^{-1} V^{T}$;
        \item $\Sigma_{0} = \left(A^{T}A\right)^{-1}$, $S_{m} = AV$ result in $x_{m} = \left.\widetilde{x}\right|_{W=AV}$, $\Sigma_{m} = \left(A^{T}A\right)^{-1} - V \left(\left(AV\right)^{T} A V\right)^{-1} V^{T}$.
    \end{enumerate}
\end{theorem}
No choice of the prior distributions in this theorem produces a useful covariance matrix. The first option leads to trivial uncertainty, while the other two are too expensive to compute. Moreover, as shown in \cite{bartels2019probabilistic} and \cite{cockayne2018bayesian}, posterior distributions of the last two choices are poorly calibrated for Krylov subspace methods. Further examination of priors reveals that they do not have free parameters, which renders uncertainty calibration impossible.

To address these problems, we propose an extension of the covariance matrix $\Sigma_{0} = VV^{T}$ that maintains the same mean of conditional distribution, but introduces a nontrivial covariance $\Sigma_{m}$. The main idea behind our construction stems from the observation made in \cite{bartels2019probabilistic}, that the first prior distribution is a probability density of random variable $x = x_0 + V v$, were $p(v)=\mathcal{N}(v|0,I)$. Perhaps it is not surprising that the posterior uncertainty is trivial, since the prior distribution puts no probability mass on the part of space where a projection method is not allowed to operate. Naturally, we seek a prior of the form $x = x_0 + V v + Y y$, $p(y)=\mathcal{N}(y|0,I)$, and restrict $Y$ to have meaningful mean and posterior covariance matrix.

In Section~\ref{section:Fixing_prior_distribution}, we completely characterize all possible choices of $Y$. Section~\ref{section:Uncertainty_calibration_for_abstract_projection_methods} contains a discussion of uncertainty calibration for abstract projection methods. A practical inexpensive construction of covariance matrix in terms of projectors is presented in Section~\ref{section:Construction_of_covariance_matrices}. In Section~\ref{section:When_probabilistic_projection_methods_are_sound} we argue that realistic Krylov subspace methods elude rigorous probabilistic interpretation. Given the popularity of Krylov subspace methods, we explain how uncertainty can be calibrated for them in Section~\ref{section:Uncertainty_calibration_for_Krylov_subspace_methods}. In Section~\ref{section:Comparison_with_BayesCG} we compare our approach with the related one, recently introduced in \cite{reid2020probabilistic}. In Section~\ref{section:Numerical_experiments} we perform a comparative study of different uncertainty calibration procedures on a several test problems that include a large family of small dense matrices, large and medium sparse matrices from SuiteSparse Matrix Collection\footnote{\url{https://sparse.tamu.edu}}, a finite-difference discretization of biharmonic equation and a PDE-constrained optimization problem.

\section{Notation}
\label{section:Notation}
In this section, we summarize some notation and definitions that we use in later parts of the paper.

For symmetric positive definite matrix $A$ we use the notation $A>0$.
For symmetric positive semidefinite matrix $A$ we use the notation $A\geq0$.

The direct sum of two matrices $A\in\mathbb{R}^{n\times m}$ and  $B\in\mathbb{R}^{l\times k}$, denoted   $A\oplus B$ is defined as
\begin{equation}
    \begin{pmatrix}
        A & 0_{n\times k}\\
        0_{l\times m} & B
    \end{pmatrix},~0_{n\times k}\in \mathbb{R}^{n\times k},~\left(0_{n\times k}\right)_{ij} = 0.
\end{equation}

For the dimension of a linear space $S$ we use the notation $\left|S\right|$.

Recall also that a pseudoinverse (Moore-Penrose inverse) of $A\in \mathbb{R}^{n\times m}$, given ${\sf rank}(A) = k$ is a matrix $A^{\dagger} = U D^{-1} V^{T}$, where columns of $U\in \mathbb{R}^{n \times k}$ are left singular vectors, columns of $V\in \mathbb{R}^{m \times k}$ are right singular vectors and diagonal matrix $D$ contains nonzero singular values $\sigma_{i}, i=1,\dots,k$, that is, $D\in\mathbb{R}^{k\times k}:D_{ij} = \sigma_{i}\delta_{ij}$ (see \cite[Lecture 4]{MR1444820}).

For matrix $Y\in\mathbb{R}^{n\times m}$ we use $Y_{\star i}, i=1,\dots, m$ to indicate column $i$ and $Y_{i\star}, i=1,\dots, n$ to indicate row $i$.

We denote indicator function for condition $x$ as
\begin{equation}
    {\sf Ind}[x] =
    \begin{cases}
        1,\text{ if condition }x\text{ holds};\\
        0,\text{ otherwise}.\\
    \end{cases}
\end{equation}

For arbitrary positive semidefinite covariance matrix $\mathbb{R}^{n\times n}\ni\Sigma = UD U^{T}\geq0$, ${\sf rank} (\Sigma) = k\leq n$, $\mathbb{R}^{k\times k}\ni D>0$, $U\in \mathbb{R}^{n\times k}$ and mean vector $\mu\in\mathbb{R}^{n}$ we define two random variables by their probability density functions. The first one is multivariate normal
\begin{equation}
    \label{multivariate_normal}
    \begin{split}
    &\mathcal{N}(x|\mu, \Sigma) = C\frac{\text{Ind}\left[U^{T}x\neq U^{T}\mu\right]}{\exp\left(\frac{1}{2}(x-\mu)^{T}\Sigma^{\dagger}(x-\mu)\right)}, \\
    &C = \frac{1}{\sqrt{(2\pi)^{k}\det{D}}}.
    \end{split}
\end{equation}
The second one is multivariate Student
\begin{equation}
    \label{Student}
    \begin{split}
    &\text{St}_{\nu}(x|\mu, \Sigma) = C\frac{\text{Ind}\left[U^{T}x\neq U^{T}\mu\right]}{\left(1+\frac{1}{\nu}(x-\mu)^{T}\Sigma^{\dagger}(x-\mu)\right)^{\frac{\nu + k}{2}}}, \\
    &C = \frac{\Gamma\left(\frac{\nu + k}{2}\right)}{\Gamma\left(\frac{\nu }{2}\right)\sqrt{(\pi\nu)^{k}\det D}}.
    \end{split}
\end{equation}
These two probability densities are constructed so that it is not possible to draw a random variable that belongs to the nullspace of the covariance matrix.

The other three distributions, i.e., inverse gamma, $F$-distribution, and $\chi^2$, are used in their standard form.

As noted above, a projection method is defined by two subspaces $\mathcal{K} = {\sf range}(V)$ and $\mathcal{L} = {\sf range}(W)$. We say that such a method is well-defined if $W^{T}AV$ is invertible. Conditions on $W$ and $V$ for a projection method to be well-defined can be found in e.g. \cite{saad2003iterative}.

In practical applications of projection methods rounding errors are important. In this article all results are given for exact arithmetic. This is not a major restriction, because we reproduce projection method exactly. This implies the whole body of known results on rounding error in projection methods can be applied as is.

\section{Fixing prior distribution}
\label{section:Fixing_prior_distribution}
In this section, we establish a sufficiently general form of $\Sigma_{0}$ that leads to nontrivial uncertainty for probabilistic projection methods. We start by proving three lemmas and then gather all results in Theorem~\ref{theorem:main_result}.
\begin{lemma}
\label{lemma:extended_prior}
    Let $V$ and $W$ lead to a well-defined projection method \eqref{projection_method}, $p(x) = \mathcal{N}\left(x|x_{0}, \Sigma_{0}\right)$, $y_{m} = S^{T}_{m}Ax$, $p(x|y_{m} = S^{T}_{m}b) = \mathcal{N}(x|x_m, \Sigma_{m})$. If we take covariance matrix $\Sigma_{0} = VV^{T} + \Psi$ and search directions $S_{m} = W$, where $\Psi$ satisfies $W^{T} A \Psi = 0$, $\Psi\geq0$, the resulting mean and covariance matrix are $x_{m} = \widetilde{x}$ from \eqref{projection_method} and $\Sigma_{m} = \Psi$.
\end{lemma}
\begin{proof}
General result \cite{bartels2019probabilistic} for mean and covariance are
\begin{equation*}
    \begin{split}
        &x_{m} = x_0 + \Sigma_{0}A^{T}S_{m}\left(S_{m}^{T}A\Sigma_{0}A^{T}S_{m}\right)^{-1}S_{m}^{T}\left(b-Ax_{0}\right),\\
        &\Sigma_{m} = \Sigma_0 - \Sigma_{0}A^{T}S_{m}\left(S_{m}^{T}A\Sigma_{0}A^{T}S_{m}\right)^{-1}S_{m}^{T}A\Sigma_{0}.
    \end{split}
\end{equation*}
Matrix $\Sigma_{0}A^{T}S_{m}$ and its transpose appear frequently in $x_{m}$ and $\Sigma_{m}$. For a chosen covariance matrix $\Sigma_0$ this combination has a simple form
\begin{equation}
    \Sigma_{0}A^{T}S_{m} = \left(VV^T + \Psi\right)A^{T}W = V \left(V^{T}A^{T}W\right),
\end{equation}
where the second equality follows from the condition $W^{T} A \Psi = 0$. Using this form of $\Sigma_{0}A^{T}S_{m}$ we find
\begin{equation}
    \left(S_{m}^{T}A\Sigma_{0}A^{T}S_{m}\right)^{-1} = \left(V^{T}A^{T} W\right)^{-1}\left(W^{T}A V\right)^{-1}.
\end{equation}
This implies that the second part of the covariance matrix simplifies as follows
\begin{equation}
    \Sigma_{0}A^{T}S_{m}\left(S_{m}^{T}A\Sigma_{0}A^{T}S_{m}\right)^{-1}S_{m}^{T}A\Sigma_{0} = VV^{T},
\end{equation}
from which we conclude that
\begin{equation}
    \Sigma_{m} = VV^T + \Psi - VV^T = \Psi.
\end{equation}
In the same vein, using
\begin{equation}
    \Sigma_{0}A^{T}S_{m}\left(S_{m}^{T}A\Sigma_{0}A^{T}S_{m}\right)^{-1}S_{m}^{T} = V \left(W^TAV\right)^{-1}W^{T}
\end{equation}
we can obtain $x_{m} = x_0 + V \left(W^TAV\right)^{-1}W^{T}(b- Ax_0)$ for the mean vector.
\end{proof}

As the following result shows, matrix $\Psi$ exists under mild conditions.
\begin{lemma}
    \label{lemma:existence}
    For invertible $A\in\mathbb{R}^{n\times n}$ and full-rank $W\in\mathbb{R}^{n\times m}$, $m\leq n$, there exists a full-rank $Y\in\mathbb{R}^{n\times k}$, $k\leq n-m$ for which $W^{T} A Y = 0$. As such, we can take $\Psi = Y G Y^{T}$ for any conformable $G>0$.
\end{lemma}
\begin{proof}
    Note, that $\left|{\sf Null}\left(W^{T}A\right)\right| = n-\left|{\sf Range}\left(A^{T}W\right)\right|=n-m$. The last equality follows from the fact that $A^{T}$ is invertible, so $W$ and $A^TW$ has the same rank. From this we conclude that there are exactly $n-m$ linearly independent vectors that span ${\sf Null}\left(W^{T}A\right)$. Stacking $k\leq n-m$ of them together we can construct $Y$.
\end{proof}
Next, we show that $\Psi$ can be chosen to have $\Sigma_{0}>0$, given $W^{T}AV$ is invertible. To demonstrate that we need to prove that for a  well-defined projection method it is always possible to supplement $m$ vectors $V_{\star i}$ with $n-m$ vectors $Y_{n-m}$ to form a basis for $\mathbb{R}^{n}$. Indeed, if this is the case, $\Sigma_{0} = V V^{T} + Y G Y^{T}>0$ since it is clearly positive semidefinite for any $G>0$, and there is no $x$ such that $x^{T} \Sigma_{0} x = 0$ because ${\sf Range}\left(V\right)\cup {\sf Range}\left(Y\right) = \mathbb{R}^{n}$.
\begin{lemma}
    \label{lemma:coverage}
     If $V$ and $W$ lead to a well-defined projection method \eqref{projection_method}, $m$ linearly independent vectors $V_{\star i}$ along with $n-m$ linearly independent $Y_{\star i}: W^{T}AY=0$ form basis for $\mathbb{R}^{n}$.
\end{lemma}
\begin{proof}
    It is easy to see that $W^{T}AV$ is invertible iff no vector from $A\mathcal{K} = {\sf Range}\left(AV\right)$ is orthogonal to $\mathcal{L}={\sf Range}\left(W\right)$.
    Vectors $Y_{\star i}$, where $i\leq n-m$, form basis for ${\sf Null}\left(W^{T}A\right)$, whereas $m$ vectors $V_{\star i}\notin {\sf Null}\left(W^{T}A\right)$, hence $V_{\star i}\in {\sf Range}\left(A^{T}W\right)$. By definition $V_{\star i}$ are linearly independent, so they form a basis for ${\sf Range}\left(A^{T}W\right)$. According to a fundamental result of linear algebra $\mathbb{R}^{n} = {\sf Null} \left(W^{T}A\right)\cup{\sf Range} \left(A^{T}W\right)$, which means columns of $V$ and $Y$ form a basis for $\mathbb{R}^{n}$.
\end{proof}

We summarize all results of this section in the following statement:
\begin{theorem}
\label{theorem:main_result}
    Let the following be true:
    \begin{enumerate}
        \item Matrix $A$ is invertible, $W, V\in \mathbb{R}^{n\times m}$ are full-rank matrices, and $\det W^{T}AV\neq 0$;
        \item Solution of $Ax = b$ is a normal random variable with probability density function $p(x) = \mathcal{N}(x|x_0, \Sigma_{0})$;
        \item Covariance matrix $\Sigma_{0}$ has a form $\Sigma_{0}=VV^{T} + Y G Y^{T}$, where ${\sf Range}\left(Y\right) = {\sf Null}\left(W^{T}A\right)$ and $G\geq0$;
        \item Random variable $y = W^{T}Ax$ represents information available to a projection method.
    \end{enumerate}
    Then under these conditions $p(x|y = W^{T}b) = \mathcal{N}(x|\widetilde{x}, Y G Y^{T})$, where $\widetilde{x}$ is defined by \eqref{projection_method}.
\end{theorem}

The proposed covariance matrix has a clear geometric meaning. It is easy to see that $x$ from Theorem~\ref{theorem:main_result} can be represented as a sum of two independent random variables $x = x_0 + V v + Y G^{1/2} y$, where $p(v) = \mathcal{N}\left(v|0, I\right)$ and $p(y) = \mathcal{N}\left(v|0, I\right)$. So, the part $V V^T$ corresponds to the vector that is sampled from ${\sf Range} (V)$, whereas the second part $Y G Y^T$ accounts for the subspace ${\sf Null} (W^TA)$ in accordance with Petrov-Galerkin condition $W^{T}(b - A \widetilde{x}) = W^{T}A(A^{-1}b - x) = 0$. Thanks to Lemma~\ref{lemma:coverage} we known that sampling $x$ we can reproduce any vector from $\mathbb{R}^{n}$, so prior distribution is suitable for an arbitrary right-hand side. Adjusting $G\geq0$ we can control how $x$ is distributed in ${\sf Range}\left(Y\right)$ (see Lemma~\ref{lemma:alignment} for a quantitative result). On the other hand it is not possible to control the distribution inside ${\sf Range}\left( V\right)$. This does not pose any problem, since as a result of projection process, the solution vector is completely defined within subspace ${\sf Range}\left( V\right)$.

\section{Uncertainty calibration for abstract projection methods}
\label{section:Uncertainty_calibration_for_abstract_projection_methods}

To be useful in practical applications (for example, in probabilistic decision theory, sensitivity analysis and others) probability density function produced by probabilistic projection methods should be meaningfully related to the actual error. In \cite{cockayne2018bayesian} authors propose a statistical criterion for uncertainty calibration: ``When the UQ is well-calibrated, we could consider $\mathbf{x}^{\star}$ [the solution $A^{-1}b$] as plausibly being drawn from the posterior distribution $\mathcal{N}(x_{m}, \Sigma_{m})$.'' Based on this statements authors suggest a test statistic $Z(x^{\star})\equiv\left\|x^{\star} - \widetilde{x}\right\|_{\Sigma_{m}^{\dagger}}^{2}\sim\chi_{n-m}^2$. In what follows we refer to $Z(x^{\star})$ as $Z-$statistic. We now show that, according to this definition, the prior proposed in Theorem~\ref{theorem:main_result} provides a perfect uncertainty calibration.
\begin{theorem}
    \label{theorem:perfect_uncertainty_calibration}
    Let $x^{\star} = x_0 + V v + Y G^{1/2}y$, where $v$ and $y$ are independent random variables, $v$ has arbitrary distribution and $p(y) = \mathcal{N}\left(y|0, I\right)$. Under conditions of Theorem~\ref{theorem:main_result}, a posterior distribution is well-calibrated:
   \begin{enumerate}
       \item $p\left(x|y = W^{T}\left(Ax_0 + AV v_{0}\right)\right) = p(x^{\star}|v=v_{0})$
        \item $\left\|x^{\star} - \widetilde{x}\right\|_{\left(YGY^{T}\right)^{\dagger}}^{2}\sim\chi_{n-m}^2$
   \end{enumerate}
\end{theorem}
\begin{proof}
\begin{enumerate}
    \item Both random variables are normal, so it is sufficient to demonstrate that first two moments are equal. Substitution of $b=AV v_0 + Ax_0$ into the definition of general projection method \eqref{projection_method} gives us $x_0 + V v_0$ which is a mean of random variable $x^{\star}$ given $v=v_{0}$. Covariance matrices coincide as a consequence of Theorem~\ref{theorem:main_result} and definition of $x^{\star}$.
    \item After the projection step, arbitrary sample of random variable $v$ is completely specified, because $Vv\in{\sf Range}(V)$. Namely, $\widetilde{x}=x_{0} + Vv$, which implies $p(x^{\star}-\widetilde{x})=\mathcal{N}(\cdot|0,YGY^{T})$. Now, since $YGY^{T}$ is positive semidefinite, it is always possible to find a full-rank matrix $X\in \mathbb{R}^{n\times k}$, where $k = {\sf rank}(YGY^{T})$ such that $XX^T$ coincides with $YGY^{T}$. It is easy to check that
    \begin{equation}
        \left(YGY^{T}\right)^{\dagger} = \left(XX^T\right)^{\dagger} = X\left(X^{T}X\right)^{-2} X^{T}.
    \end{equation}
    Since $x^{\star}-\widetilde{x} = X \delta$, where $\delta$ is a standard multivariate normal random variable, we can find that test statistic
    \begin{equation}
        \left\|x^{\star} - \widetilde{x}\right\|_{\left(YGY^{T}\right)^{\dagger}}^{2} = \delta^TX^T X\left(X^{T}X\right)^{-2} X^{T} X \delta = \delta^T \delta
    \end{equation}
    follows $\chi^2_{n-m}$ distribution.
\end{enumerate}
\end{proof}
Note, that this result is also correct for all priors proposed in Theorem~\ref{theorem:prior_art}. This is because all methods are fully Bayesian when $W$ and $V$ do not depend on $x$. As we discuss in Section~\ref{section:When_probabilistic_projection_methods_are_sound}, this is not true for Krylov subspace methods like CG and GMRES.

Having a well-calibrated posterior probability, we turn to the choice of a prior distribution. Since with $G$, we can always perform a change of basis in ${\sf Null}\left(W^{T}A\right)$; we consider it to be fixed and describe how the rescaling of basis vectors influences an error vector.
\begin{lemma}
    \label{lemma:alignment}
    Let in addition to conditions of Theorem~\ref{theorem:main_result} columns of matrix $Y$ be orthonormal, and the exact solution be $x^{\star} = x_0 + V \delta_{1} + Y G^{1/2}\delta$, where $\delta_1, \delta$ are standard multivariate normal random variables. The choice $G = s^2 I_{p\times p}\oplus I_{\left(n-m-p\right)\times \left(n-m-p\right)}$, $s\in\mathbb{R}$ leads to $\cos\left(\theta\right) = 1\big/\left(1+\frac{n-m-p}{s^2p}z\right)$, where $\theta$ is an acute angle between the error $\widetilde{e} = x^{\star} - \widetilde{x}$ and ${\sf span}\left\{Y_{\star i}:i=1,\dots,p\right\}$; $z$ is $F$-distributed with numerator $n-m-p$ and denominator $p$ (see Figure~\ref{fig:directional_UQ} for geometric interpretation).
\end{lemma}
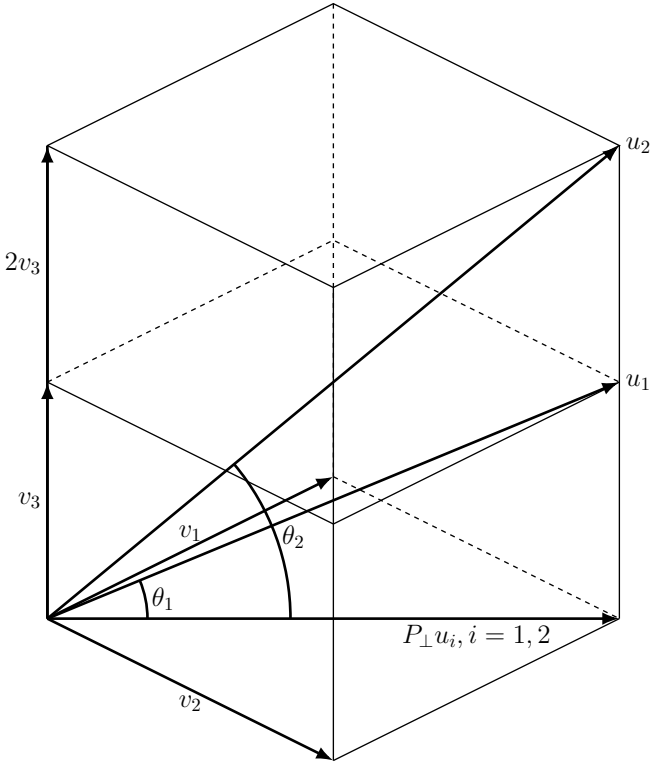
\begin{figure}[t]
    \centering
    \resizebox{0.5\textwidth}{!}{\input{directional_UQ.tex}}
    \caption{The figure demonstrates how the acute angle $\theta_{i},~i=1,2$ between subspace spanned by $v_1$ and $v_2$ and $u_i,~i=1, 2$ depend on the vector $u_{i} = v_1 + v_2 + iv_3$. The angles can be computed as $\cos(\theta_{i}) = u_{i}^{T}P_{\perp} u_{i}\big/u_{i}^T u_{i}$. Lemma~\ref{lemma:alignment} is a probabilistic counterpart of this situation. Namely, by rescaling eigenvectors of covariance matrix one can influence the distribution of the angle between the error and a given subspace.}
    \label{fig:directional_UQ}
\end{figure}
\begin{proof}
    Since $x^{\star} = x_0 + V\delta_1 + Y G^{1/2}\delta$, where $\delta_1$ and $\delta$ are independent standard multivariate normal variables, an error $\widetilde{e}$ after the projection step \eqref{projection_method} is $Y G^{1/2}\delta$.
    Using the definition of the acute angle $\theta$ (see Figure~\ref{fig:directional_UQ}), and orthogonal projector $P_{\perp} = \sum_{i=1}^{p}Y_{\star i} Y_{\star i}^{T}$ on subspace spanned by vectors $Y_{\star i},~i=1,\dots,p$ we can show that
    \begin{equation}
        \begin{split}
        \cos(\theta)&= \frac{\widetilde{e}^{T}\left(\sum_{i=1}^{p} Y_{\star i}Y_{\star i}^{T}\right) \widetilde{e}}{\widetilde{e}^T \widetilde{e}} \\ &=\frac{s^2\sum_{i=1}^{p}\delta_{i}^{2} }{\sum_{i=p+1}^{n-m-p}\delta_{i}^{2}+s^2\sum_{i=1}^{p}\delta_{i}^{2}} = \frac{1}{1 + \frac{\chi^2_{n-m-p}}{s^2\chi^2_{p}}}
        \end{split}
    \end{equation}
    Since $z = \left(p\chi^2_{n-m-p}\right)/\left((n-m-p)\chi^2_{p}\right)$ is $F$-distributed (see Chapter~13 in \cite{MR3642449}) the proof is complete.
\end{proof}
With this result we can easily construct probabilistic bounds. For example, identity $P\left(\cos(\theta)\geq 1-\epsilon\right) = P\left(z\leq s^2p\epsilon/\left((1-\epsilon)(n-m-p)\right)\right)$ allows to choose $s$ that guaranties $\widetilde{e}$ to be located within a $p$-dimensional subspace with prescribed probability.
\section{Construction of covariance matrices}
\label{section:Construction_of_covariance_matrices}
So far, we discussed only a general form of a covariance matrix. The most straightforward way to construct it explicitly is to compute a basis for ${\sf Null} \left(W^{T}A\right)$ with SVD and choose positive semidefinite $G$ according to some criteria. This can be problematic for two reasons. First, SVD incurs additional $O\left(n m^2\right)$ floating-point operations \cite{MR1444820}. Depending on the situation, this can be manageable. The second and more serious problem is that we need to store a dense $n \times \left(n-m\right)$ matrix. Iterative methods are useful only when $A$ is sparse and large, so as a rule, we do not have the luxury to store $\left(n-m\right)$ vectors forming a basis for ${\sf Null} \left(W^{T}A\right)$. The following result resolves these issues.
\begin{theorem}
    \label{theorem:simplified_uncertainty}
    Let conditions of  Theorem~\ref{theorem:main_result} be fulfilled. For $P_1 = I - V \left(W^{T}AV\right)^{-1}W^{T}A$ the following statements are true:
    \begin{enumerate}
        \item Matrix $P_1$ is a projection operator.
        \item ${\sf Range}(P_1) = {\sf Null}\left(W^{T}A\right)$
        \item General form of covariance matrix from Lemma~\ref{lemma:existence} is $\Sigma_{0} = VV^{T} + P_1 G P_1^{T}$, $G\geq 0$.
    \end{enumerate}
\end{theorem}
\begin{proof}
    \begin{enumerate}
        \item It is enough to demonstrate that $I-P_1$ is a projection operator. Indeed, if this is the case, $P_1$ is a projection operator too since $(I-P_1)^2 = I - P_1$ implies that $P_1^2 = P_1$. Using $I-P_1 = \left(W^{T}AV\right)^{-1}W^{T}A$ for $\left(I-P_1\right)^2$ we find
        \begin{equation}
        \begin{split}
            V\left(W^{T}AV\right)^{-1}W^{T}AV\left(W^{T}AV\right)^{-1}W^{T}A \\
            =  V\left(W^{T}AV\right)^{-1}W^{T}A = I - P_1,
        \end{split}
        \end{equation}
        so $I-P_1$ is a projection operator.
        \item
        It is easy to see that $W^{T}A P_1 = 0$. Indeed,
        \begin{equation}
            \begin{split}
            W^{T}A\left(I - V \left(W^{T}AV\right)^{-1}W^{T}A\right) \\
            = W^{T}A - W^{T}A = 0
            \end{split}
        \end{equation}
        From $W^{T}A P_1 = 0$ we have ${\sf Range}(P)\subseteq {\sf Null}(W^{T}A)$. On the other hand $W^{T}Ax = 0 \Rightarrow P_1x = x$, so ${\sf Null}(W^{T}A)\subseteq {\sf Range}(P)$. From two inclusions we conclude that ${\sf Range}(P)={\sf Null}(W^{T}A)$.
        \item Any $\Psi\geq 0$ from Lemma~\ref{lemma:existence} has a form $Y Y^T$ where columns of $Y$ belong to ${\sf Null}(W^{T}A)$. This fact follows from spectral decomposition of $\Psi$, $\Psi\geq 0$ and $W^{T}A \Psi = 0$. Since ${\sf Range}(P_1) = {\sf Null}\left(W^{T}A\right)$ we know that $P_1Y = Y$. This allows us to take $G = YY^T$ for which the covariance matrix reads
        \begin{equation}
            \begin{split}
            \Sigma_{0} &= VV^T + P_1Y \left(P_1 Y\right)^{T} \\ &= VV^T + YY^T = VV^T + \Psi.
            \end{split}
        \end{equation}
        So with the appropriate choice of $G$ we can reproduce arbitrary covariance matrix from Lemma~\ref{lemma:existence}.
    \end{enumerate}
\end{proof}
We would like to point out that it is natural to use projector $P_1$ to quantify uncertainty. It is known from general theory of iterative methods (see Chapter~2 from \cite{MR3495481}) that linear iteration of the form $x^{(n+1)} = x^{(n)} + N[A]r^{(n)}\equiv M[A]x^{(n)} + N[A]b$, where $M[A]$ and $N[A]$ are matrices depending on $A$ such that the consistency condition $M[A] + N[A]A = I$ holds. In our case $N[A] = V\left(W^{T}AV\right)^{-1}W^{T}$ approximates $A^{-1}$, and $P_1 = M[A]  = I - N[A]A$ quantifies how well this is done.

To compute projection operator $P_1$ from Theorem~\ref{theorem:simplified_uncertainty}, one need not perform more complex operations that are required for projection method itself: matrices $W$ and $V$ are available as a byproduct of Arnoldi or Lanczos processes and $W^{T}AV$ usually has a special form (Hessenberg or tridiagonal). Moreover, to store $P$, we need to keep matrices $W$, $V$, and $\left(W^{T}AV\right)^{-1}$, that is $2nm + m^2$ floating-point numbers in the worst case, which is much better than $n^2-mn$ in situations when $m\ll n$.

Covariance matrix in Theorem~\ref{theorem:simplified_uncertainty} contains projection operator $P_1$ which is not orthogonal. Later we will see that orthogonal projectors are more suitable in the context of statistical inference, so we formulate a result similar to Theorem~\ref{theorem:simplified_uncertainty} but with an orthogonal projector.
\begin{theorem}
\label{theorem:orthogonal_projector}
    Let $P_2 = Y \left(Y^T Y\right)^{-1} Y^T$, where columns of $Y$ are $k=n-m$ linearly independent vectors from ${\sf Null}\left(W^T A\right)$. If $W$ and $V$ result in a well-defined projection method, the following is true:
    \begin{enumerate}
        \item $P_2$ is an orthogonal projector on ${\sf Null}\left(W^T A\right)$.
        \item Covariance matrix $\Sigma_{0} = VV^T + P_{2}GP_{2}^T$, $G\geq 0$, leads to a posterior $\mathcal{N}(\cdot|\widetilde{x}, P_2GP_2^T)$, under linear observations and conditions defined in Theorem~\ref{theorem:main_result}.
    \end{enumerate}
\end{theorem}
\begin{proof}
    \begin{enumerate}
        \item $P^2_2 = Y \left(Y^T Y\right)^{-1} Y^TY \left(Y^T Y\right)^{-1} Y^T = P_2$, so $P_2$ is a projection operator. Next, $P_2^{T} = P_2$ so $P_2$ is an orthogonal projector. Finally, ${\sf Range}(P_2) = {\sf Null}\left(W^T A\right)$ by definition of $Y$.
        \item From ${\sf Range}(P) = {\sf Null}(W^{T}A)$ it follows that $W^{T}A P_2 = 0$, and $\Sigma_{0}A^{T}W = V V^{T}A^{T}W$. Since the proof of Lemma~\ref{lemma:extended_prior} relies only on the fact that $W^{T}A \Psi = 0$, we can substitute $\Psi$ by $P_2$ and obtain the same result. With that we conclude that the posterior distribution has a probability density $\mathcal{N}(\cdot|\widetilde{x}, P_2GP_2^T)$.
    \end{enumerate}
\end{proof}

Note that to compute $P_2$ one need no explicitly form the orthonormal basis for ${\sf Null}\left(W^T A\right)$, which is not feasible in typical practical situations when $n\gg 1$ and $m\ll n$. In place of that, one can use $\widetilde{Y}\in\mathbb{R}^{n \times m}$ with columns such that ${\sf Range}\left(\widetilde{Y}\right) = {\sf Range}\left(A^T W\right)$. Since ${\sf Range}\left(A^T W\right) \perp {\sf Null}\left(W^T A\right)$ we conclude that $P_2 = I - \widetilde{Y}\left(\widetilde{Y}^{T}\widetilde{Y}\right)^{-1}\widetilde{Y}^{T}$. Unlike $Y$, computation of $\widetilde{Y}$ is feasible. Moreover, for some projection method $\widetilde{Y}$ can be available as a byproduct of the method itself. For example, vectors from ${\sf Range}\left(A^T W\right)$ are available in case of Lanczos biorthoganolization (see \cite[Subsection 7.2]{saad2003iterative}). These vectors are discarded when only the solution of the linear system is of interest, however as we see from Theorem~\ref{theorem:orthogonal_projector} they can be used to construct a covariance matrix. Conjugate gradient iteration provides the other example. In this case $A>0$ and $W = V$, so the residuals can be used to form orthonormal basis for ${\sf Range}\left(A V\right)$.

\section{When probabilistic projection methods are sound}
\label{section:When_probabilistic_projection_methods_are_sound}
The validity of Theorem~\ref{theorem:prior_art} and Bayesian conjugate gradient Method proposed in \cite{cockayne2018bayesian}, as well as all results of the present paper, depend on the assumption that the joint distribution of $x$ and $y_{m}$ is a multivariate normal. This fact can be shown via computation of characteristic function if search directions $S_{m}$ and prior covariance matrix $\Sigma_{0}$ are independent of $x$. When Krylov subspace $\mathcal{K}_{m}\left(A, b\right)$ is used to build $S_{m}$, as it is done in almost all Krylov subspace methods, information $y_{m}$ becomes a nonlinear function of $x$, and the joint distribution of $y_{m}$ and $x$ is not a multivariate normal. This implies that algorithms based on Theorem~\ref{theorem:prior_art} and Bayesian conjugate gradient cannot stand as probabilistic Krylov subspace methods. Moreover, even when $S_{m}$ is unrelated to $x$, as in the Lanczos biorthogonalisation algorithm, $V$, that still depends on $x$, is not allowed to appear in prior covariance matrix $\Sigma_0$. These restrictions render probabilistic Krylov projection methods incorrect. We can think of three possible solutions to this problem.

The first solution is to focus on projection methods that do not use $\mathcal{K}_{m}\left(A, b\right)$ to construct approximate solution. For example, a two-grid operator in the Algebraic Multigrid (AMG) framework has the same form as a projection method \eqref{projection_method}, given $V$ is a matrix of interpolation operator and $W$ is a matrix of restriction operator. The same is true for Gauss-Seidel method, which is equivalent to the sequence of projection steps with $\mathcal{L}=\mathcal{K}={\sf span}\left\{e_{i}\right\}$ repeated for $i=1,\dots,n$ until convergence.

Another way is to use Arnoldi or Lancsoz processes to build basis in $\mathcal{K}_{m}\left(A, \rho\right)$, where $\rho$ is independent of $b$. For this kind of projection processes, probabilistic methods are rigorously justified. On the downside, there are few theoretical results and estimations available from numerical linear algebra. One can also expect a deterioration of the convergence rate. In addition to that, memory-friendly algorithms like Conjugate Gradient should be rederived (if this is possible at all), because they explicitly rely on the fact that the first search direction is parallel to an initial residual vector.

Finally, it is possible to apply the results obtained under the assumption that $W$ and $V$ are independent of $x$ to actual Krylov subspace methods and try to tune prior probability to get well-calibrated uncertainty. We consider this option in the next section.

\section{Uncertainty calibration for Krylov subspace methods}
\label{section:Uncertainty_calibration_for_Krylov_subspace_methods}
For Krylov subspace methods, uncertainty is poorly calibrated. In the present section we put forward a statistical procedure that allows us to adjust a single scalar parameter in such a way, that $Z-$statistic as well as $S-$statistic (to be defined) are well-calibrated.

Before the main results we prove the following supplementary lemma.
\begin{lemma}
    \label{lemma:hierarchical_modelling}
    Let $p(s|\alpha, \beta) = {\sf IG}\left(s|\alpha, \beta\right)$ be the inverse-gamma distribution, and $p(x|s, \Sigma, \mu) = \mathcal{N}\left(x|\mu, s\Sigma\right)$, $\Sigma\geq 0$, then
    \begin{equation}
        \begin{split}
            p(x|\Sigma, \mu, \alpha, \beta) = \int dx~p(x|s, \Sigma, \mu)p(s|\alpha, \beta) \\
            = {\sf St}_{2\alpha}\left(x\left|\mu, \frac{\beta}{\alpha}\Sigma\right.\right).
        \end{split}
    \end{equation}
\end{lemma}
\begin{proof}
    The result is a slight generalisation of a standard Bayesian hierarchical modelling for multivariate normal distribution \cite{bernardo2009bayesian}. Using definition of inverse-gamma distribution and probability density function of multivariate normal distribution \eqref{multivariate_normal} we obtain
    \begin{equation}
        \begin{split}
            &p(x|s, \Sigma, \mu)p(s|\alpha, \beta) = \frac{\beta^{\alpha}}{\Gamma(
            \alpha)}s^{-(\alpha+1)}\exp(-\beta\big/s) \\
            &\frac{\exp\left(-(x-\mu)^T\Sigma^{\dagger}(x-\mu)\big/(2s)\right)}{\left(2\pi s\right)^{k\big/2}\sqrt{\det D}}{\sf Ind}\left[U^Tx\neq U^{T}\mu\right] \\
            &= \frac{\beta^{\alpha}}{\Gamma(
            \alpha)} \frac{\Gamma(
            \overline{\alpha})}{\overline{\beta}^{\overline{\alpha}}}\frac{{\sf IG}\left(s|\overline{\alpha}, \overline{\beta}\right){\sf Ind}\left[U^Tx\neq U^{T}\mu\right]}{(2\pi)^{k\big/2}\sqrt{\det D}},
        \end{split}
    \end{equation}
    where $\overline{\alpha}=\alpha + k\big/2$, $\overline{\beta} = \beta + (x-\mu)^T\Sigma^{\dagger}(x-\mu)\big/2$. Probability density function ${\sf IG}\left(s|\overline{\alpha}, \overline{\beta}\right)$ disappears after integration, and it is easy to see that the remaining factors form ${\sf St}_{2\alpha}\left(x\left|\mu, \frac{\beta}{\alpha}\Sigma\right.\right)$ defined in \eqref{Student}.
\end{proof}

The first result is based on the rescaling of the full covariance matrix from Theorem~\ref{theorem:main_result} as proposed in \cite{cockayne2018bayesian}.
\begin{lemma}
    \label{lemma:cheap_UQ}
    Let conditions of Theorem~\ref{theorem:main_result} be fulfilled. For covariance matrix $\Sigma_{0} = s\left(VV^{T} + \Psi\right)$, $s>0$, $W^{T}A\Psi=0$, $\Psi\geq 0$; and prior $p(s|\alpha, \beta)={\sf IG}\left(s|\alpha, \beta\right)$ the following is true:
    \begin{enumerate}
        \item Probability density function $p\left(s|W^{T}Ax = W^{T}b\right)$ is the inverse-gamma distribution with parameters $\widetilde{\alpha} = \alpha + m/2$, $\widetilde{\beta} = \beta + \delta^{T}\delta / 2$, $\delta = \left(W^{T}AV\right)^{-1}W^{T}\left(b - Ax_{0}\right)$.
        \item Predictive distribution for $x|W^{T}Ax = W^{T}b$ is multivariate Student distribution ${\sf St}_{2\widetilde{\alpha}}\left(x|\widetilde{x}, \frac{\widetilde{\beta}}{\widetilde{\alpha}}\Psi\right)$.
    \end{enumerate}
\end{lemma}
\begin{proof}
    \begin{enumerate}
        \item We define random variable $z = W^{T}Ax$. Since $x = x_0 + s^{1/2}V \delta_{1} + s^{1/2}\Psi^{1/2} \delta_{2}$, where $\delta_{i},~i=1,2$ are independent standard multivariate normal random variables and $W^{T}A\Psi^{1/2} = 0$, probability density function for z reads
        \begin{equation}
             p(z) = \mathcal{N}\left(z|W^{T}A x_0,  sW^{T}AV \left(W^{T}AV\right)^T\right).
        \end{equation}
       Using definition of posterior distribution we find
        \begin{equation}
            \begin{split}
                &p\left(s|W^{T}Ax = W^{T}b\right) \propto p(z = W^{T}b) {\sf IG}(s|\alpha, \beta) \\
                &\propto s^{-m\big/2}\exp\left(-\delta^T\delta\big/(2s)\right) s^{-(\alpha+1)}\exp\left(-\beta\big/s\right),
            \end{split}
        \end{equation}
        where $\delta = \left(W^{T}AV\right)^{-1}W^{T}\left(b - Ax_{0}\right)$.
        From the last line we can identify parameters of the posterior distribution $\widetilde{\alpha} = \alpha + m/2$, $\widetilde{\beta} = \beta + \delta^{T}\delta / 2$.
        \item Predictive distribution is
        \begin{equation}
            \begin{split}
                &p\left(x\left|W^{T}Ax = W^{T}b\right.\right) \\
                &= \int ds~p\left(x\left|W^{T}Ax = W^{T}b, x_0, s\Sigma_{0}\right.\right) p(s|\alpha, \beta),
            \end{split}
        \end{equation}
        where the first factor under the integral is multivariate normal $\mathcal{N}\left(x|\widetilde{x}, s
        \Psi\right)$ (see Theorem~\ref{theorem:simplified_uncertainty}), and the second is ${\sf IG}(s|\widetilde{\alpha}, \widetilde{\beta})$.
        Using the result from Lemma~\ref{lemma:hierarchical_modelling} we obtain ${\sf St}_{2\widetilde{\alpha}}\left(x|\widetilde{x}, \frac{\widetilde{\beta}}{\widetilde{\alpha}}\Psi\right)$ as a predictive distribution.
    \end{enumerate}
\end{proof}

Lemma~\ref{lemma:cheap_UQ} is straightforward from the point of view of the implementation, because approximate solution \eqref{projection_method} is $\widetilde{x}= x_{0} + V\delta$, where $\delta = \left(W^{T}AV\right)^{-1}W^{T}\left(b - Ax_{0}\right)$, scalar $\left\|\delta\right\|_2^2$, required for uncertainty calibration, can be readily computed for arbitrary projection method. Common factor $s$ appears in Lemma~\ref{lemma:cheap_UQ} because if we take $\Sigma_{0} = VV^{T} + s\Psi$, posterior distribution for the scale $p\left(s|W^{T}Ax = W^{T}b\right)$ coincides with ${\sf IG}\left(s|\alpha, \beta\right)$, that is available information is insufficient to fix the scale. Since a scale of an error can be completely unrelated to the $L_{2}$ norm of projection of $A^{-1}b$ on $V$, additional information can be valuable to tune $s$. This is explored in the following result.
\begin{lemma}
    \label{lemma:expensive_UQ}
    Let conditions of Theorem~\ref{theorem:orthogonal_projector} be fulfilled and $G = sI$ for $s>0$, so $\Sigma_{0} = VV^T + sP_2$, the solution is a multivariate normal variable $p(x) = \mathcal{N}(x|x_0, \Sigma_{0})$. For a prior distribution $p(s|\alpha, \beta) = {\sf IG}(s|\alpha, \beta)$ and i.i.d. observations $X_{\star i},~i=1,\dots, k$ of random variable $P_1(x-x_0)$ (here $P_1$ is as in Theorem~\ref{theorem:simplified_uncertainty}) the following is true:
    \begin{enumerate}
        \item Posterior distribution of $s|X$ is ${\sf IG}(s|\widetilde{\alpha}, \widetilde{\beta})$, where $\widetilde{\alpha} = \alpha + k(n-m)\big/2$, $\widetilde{\beta} = \beta + {\sf tr}\left(X^{T}X\right)/2$.
        \item Predictive distribution of $x|W^TA x = W^Tb, X$ is multivariate Student ${\sf St}_{2\widetilde{\alpha}}\left(x\left|\widetilde{x}, \frac{\widetilde{\beta}}{\widetilde{\alpha}}P_2\right.\right)$.
    \end{enumerate}
\end{lemma}
\begin{proof}
    \begin{enumerate}
        \item We define random variable $z = P_1(x-x_0)$. To find probability density function of $z$ we use three facts. First, $P_1 V = 0$, which follows from definition of $P_1$. Second, because ${\sf Range}(P_2) = {\sf Range}(P_1)$, we conclude that $P_1 P_2 = P_2$. Finally, $x = x_0 + s^{1/2}V \delta_{1} + s^{1/2}P_2 \delta_{2}$, where $\delta_{i},~i=1,2$ are independent standard multivariate normal distributions. Using these three facts we find $p(z) = \mathcal{N}\left(z|0, s P_2\right)$. Now, it is easy to find a posterior distribution
        \begin{equation}
            \begin{split}
                & p(s|X) \propto \left(\prod_{i=1}^{k}p\left(z_{i}=X_{\star i}\right)\right)p(s|\alpha, \beta) \propto \exp(-\beta\big/s)\\
                & s^{-(\alpha+1)}s^{-k(n-m)/2} \exp\left(-\sum_{i=1}^{k}X_{\star i}^T P_{2}^{\dagger}X_{\star i}\big/(2s)\right).
            \end{split}
        \end{equation}
        Because $P_2$ is orthogonal projector $P_2^{\dagger} = P_2$. In addition to that, $X_{\star i}$ belongs to ${\sf Range}(P_1)$, so each term of the quadratic form simplifies $X_{\star i}^T P_{2}^{\dagger}X_{\star i} = X_{\star i}^T X_{\star i}$. Using the definition of inverse-gamma distribution we can identify new parameters $\widetilde{\alpha} = \alpha + k(n-m)\big/2$, $\widetilde{\beta} = \beta + {\sf tr}\left(X^{T}X\right)/2$.
        \item Predictive distribution is
        \begin{equation}
            \begin{split}
                &p\left(x\left|W^{T}Ax = W^{T}b, X\right.\right) \\
                &= \int ds~p\left(x\left|W^{T}Ax = W^{T}b, x_0, s\Sigma_{0}\right.\right) p(s|X),
            \end{split}
        \end{equation}
        where the first factor under the integral is multivariate normal $\mathcal{N}\left(x|\widetilde{x}, s P_2\right)$ (this follows from Theorem~\ref{theorem:orthogonal_projector} with $G = sI$), and the second is ${\sf IG}\left(s|\widetilde{\alpha}, \widetilde{\beta}\right)$.
        Using the result from Lemma~\ref{lemma:hierarchical_modelling} we confirm that the predictive distribution is ${\sf St}_{2\widetilde{\alpha}}\left(x|\widetilde{x},\frac{\widetilde{\beta}}{\widetilde{\alpha}}P_2\right)$.
    \end{enumerate}
\end{proof}

The reason why we take $P_1x$ as an additional observation to fix the scale is that an exact solution has a representation $x^{\star} = (I-P_1)x^{\star} + P_1x^{\star}$. If $x_0=0$ the first term $(I-P_1)x^{\star} = \widetilde{x}$, so $P_1x^{\star}$ is an error. To collect independent sample $x_{P}$ we need to run the same projection method second time, starting from a sample $x^{\star}$ from a prior distribution that we presume to be available. As a result, application of Lemma~\ref{lemma:expensive_UQ} doubles (for $k=1$) numerical costs of any projection method. This is summarized in Algorithm~\ref{algorithm:UQ_calibration}.

\begin{algorithm}
    \caption{Uncertainty calibration.}
    \label{algorithm:UQ_calibration}
    \begin{algorithmic}[1]
        \STATE \textbf{Input:} distributions for exact solution $p(x^{\star})$, $x^{\star}\in\mathbb{R}^{n}$; a projection method $V, W \leftarrow {\sf Proj}\left(A, b, m\right)$; a number of search directions $m$; parameters of inverse-gamma distribution $\alpha, \beta$; a number of observations $k$; ${\sf statistic}$ either $S$ or $Z$.
        \STATE \textbf{Output:} modified parameters of inverse-gamma distribution $\widetilde{\alpha}, \widetilde{\beta}$.
        \\\hrulefill
        \STATE $\widetilde{\alpha} = \alpha + k(n-m)\big/2$
        \STATE $\widetilde{\beta} = \beta$
        \FOR{$i=1:k$}
            \STATE $x^{\star}\sim p(x^{\star})$
            \STATE $b = Ax^{\star}$
            \STATE $V, W \leftarrow  {\sf Proj}\left(A, b, m\right)$
            \STATE $x^{\star} \leftarrow \left(I - V \left(W^{T}AV\right)^{-1}W^{T}A\right)x^{\star}$
            \IF{${\sf statistic} = Z$}
                \STATE $\delta = \left(x^{\star}\right)^T \left(x^{\star}\right)$
            \ENDIF
            \IF{${\sf statistic} = S$}
                \STATE $\delta = \left(x^{\star}\right)^T A\left(x^{\star}\right)$
            \ENDIF
            \STATE $\widetilde{\beta} = \widetilde{\beta} + \delta/2$
        \ENDFOR
    \end{algorithmic}
\end{algorithm}

Note, that prior from Lemma~\ref{lemma:expensive_UQ} leads to simple form of $Z-$statistic $Z(x^{\star}) = \left\|x^{\star} - x_{m}\right\|_{\Sigma_{m}^{\dagger}}$, where $x^{\star} = x_0 + V\delta_1 + s^{1\big/2}P_2\delta_2$ is an exact solution, $\Sigma_{m}^{\dagger}$ and $x_{m}$ are posterior covariance matrix and posterior mean vector respectively and $\delta_{i},~i=1,2$ are standard multivariate normal random variables. Indeed, because $P_2$ is an orthogonal projector $P_{2}^{\dagger}=P_{2}$. Moreover, an error $x^{\star}-x_{m}$ belongs to ${\sf Range}(P_2) = {\sf Range}(P_1)$ which follows from the fact that $x^{\star} - x_{m} = P_1(x^{\star} - x_{0})$. So we can conclude that test statistic is simply a squared $L_2$ norm of the error $\left\|x^{\star}-x_{m}\right\|_2^2$. In light of this observation, Algorithm~\ref{algorithm:UQ_calibration} simply samples an error from a known $x^{\star}$ and use its squared $L_2$ norm to estimate an error for a given right-hand side $b$ for which the exact solution is unknown.

Both Lemma~\ref{lemma:cheap_UQ} and Lemma~\ref{lemma:expensive_UQ} are designed for test $Z-$statistic. Recently \cite{reid2020probabilistic} propose a different test statistic $S(x) = \left(x-x_{m}\right)A\left(x-x_{m}\right)$, where $x$ is drawn from a posterior distribution given linear observations as in Theorem~\ref{theorem:main_result}. In what is following we call this random variable $S-$statistic. To calibrate the scale for $S-$statistic we use the following result.

\begin{lemma}
    \label{lemma:covariance_for_Reid}
    Let $A>0$, $W=V$, columns of $Y$ in Theorem~\ref{theorem:main_result} are $A-$orthogonal, i.e., $Y^TAY = I$ and $G=sI,~s>0$. Let $Z_{\star i}, i=a,\dots, k$ be a set of i.i.d. observations of random variable $A^{1/2}P_1(x-x_0)$ (here $P_1$ is as in Theorem~\ref{theorem:simplified_uncertainty}). For the prior distribution $p(s)={\sf IG}(s|\alpha, \beta)$ under condition of Theorem~\ref{theorem:main_result} the following is true:
    \begin{enumerate}
        \item Posterior distribution of $s|Z$ is ${\sf IG}\left(s|\widetilde{\alpha}, \widetilde{\beta}\right)$, $\widetilde{\alpha}=\alpha + k(n-m)\big/2$, $\widetilde{\beta}=\beta + {\sf tr}\left(Z^T Z\right)\big/2$.
        \item Predictive distribution of $x|W^TA x = W^Tb, Z$ is multivariate Student ${\sf St}_{2\widetilde{\alpha}}\left(x\left|\widetilde{x}, \frac{\widetilde{\beta}}{\widetilde{\alpha}}YY^T\right.\right)$.
    \end{enumerate}
\end{lemma}
\begin{proof}
    \begin{enumerate}
        \item We define random variable $z = A^{1/2}P_1(x-x_0)$. To find probability density function of $z$ we use three facts. First, $P_1 V = 0$, which follows from definition of $P_1$. Second, because ${\sf Range}(P_1) = {\sf Range}(Y)$, we conclude that $P_1 Y = Y$. Finally, $x = x_0 + V \delta_{1} + s^{1/2}Y \delta_{2}$, where $\delta_{i},~i=1,2$ are independent standard multivariate normal random variables. Using these three facts we find a probability density function $p(z) = \mathcal{N}\left(z|0, s A^{1/2}YY^{T}A^{1/2}\right)$. It is easy to see that $P_3 = A^{1/2}YY^{T}A^{1/2}$ is an orthogonal projector. Indeed, from $A-$orthogonality we conclude that
        \begin{equation}
            P_3^2 = A^{1/2}YY^{T}AYY^{T}A^{1/2} = P_3.
        \end{equation}
        The orthogonality $P_3^T = P_3$ follows from $A^T=A$. Now, it is easy to find a posterior distribution
        \begin{equation}
            \begin{split}
                & p(s|Z) \propto \left(\prod_{i=1}^{k}p\left(z_{i}=Z_{\star i}\right)\right)p(s|\alpha, \beta) \propto \exp(-\beta\big/s)\\
                & s^{-(\alpha+1)}s^{-k(n-m)/2} \exp\left(-\sum_{i=1}^{k}Z_{\star i}^T P_{3}^{\dagger}Z_{\star i}\big/(2s)\right).
            \end{split}
        \end{equation}
        Using that $P_{3}^{\dagger} = P_{3}$ ($P_3$ is an orthogonal projector) and that $z \in {\sf Range}(A^{1/2}Y)$ (this follows from ${\sf Range}(P_1)={\sf Range}(Y)$) we simplify quadratic form $Z_{\star i}^T P_{3}^{\dagger}Z_{\star i} = Z_{\star i}^T Z_{\star i}$. Using the definition of inverse-gamma distribution we can identify new parameters $\widetilde{\alpha}=\alpha + k(n-m)\big/2$, $\widetilde{\beta}=\beta + {\sf tr}\left(Z^T Z\right)\big/2$.
        \item Predictive distribution is
        \begin{equation}
            \begin{split}
                &p\left(x\left|W^{T}Ax = W^{T}b, Z\right.\right) \\
                &= \int ds~p\left(x\left|W^{T}Ax = W^{T}b, x_0, s\Sigma_{0}\right.\right) p(s|Z),
            \end{split}
        \end{equation}
        where the first factor under the integral is multivariate normal $\mathcal{N}\left(x|\widetilde{x}, s YY^T\right)$ (this follows from Theorem~\ref{theorem:main_result} with $G = sI$), and the second is ${\sf IG}\left(s|\widetilde{\alpha}, \widetilde{\beta}\right)$.
        Using the result from Lemma~\ref{lemma:hierarchical_modelling} we confirm that the predictive distribution is ${\sf St}_{2\widetilde{\alpha}}\left(x|\widetilde{x},\frac{\widetilde{\beta}}{\widetilde{\alpha}}YY^T\right)$.
    \end{enumerate}
\end{proof}

The uncertainty calibration is summarized in Algorithm~\ref{algorithm:UQ_calibration}.
As explained in the next result, the covariance matrix from Lemma~\ref{lemma:covariance_for_Reid} leads to simple $S-$statistic.
\begin{lemma}
    \label{lemma:simplified_S_statistic}
    Under conditions of Lemma~\ref{lemma:covariance_for_Reid} distribution of $S(x) = (x - \widetilde{x})A(x - \widetilde{x})$ is the same as distribution of $s\chi_{n-m}^2$.
\end{lemma}
\begin{proof}
    Distribution of $x - \widetilde{x}$ is $\mathcal{N}(\cdot|0, sYY^T)$, so $x - \widetilde{x} = s^{1/2} Y \delta_1$, where $\delta_1$ is standard multivariate normal variable. Using this we find $S(x) = s \delta_1^TY^TAY\delta_{1} = s \delta_1^T\delta_{1}$, where the last equality follows from $A-$orthogonality of columns of $Y$.
\end{proof}

Note, that $z^Tz$ from Lemma~\ref{lemma:covariance_for_Reid} is an independent sample from $e^TA e$, where $e$ is a current error vector. So, if $\alpha=\beta=0$, mean value of $s$ is approximately $\sum_{i=1}^{k}e_{i}^T Ae_{i}\big/(k(n-m))$, so $S-$statistic takes a form
\begin{equation}
    S(x) \simeq \frac{1}{k}\left(\sum_{i=1}^{k}e_{i}^T Ae_{i}\right) \frac{\chi_{n-m}^{2}}{(n-m)}.
\end{equation}
Since $\mathbb{E}\left[\chi_{n-m}^{2}\right] = n-m$ we can expect that $S-$statistic is well calibrated.

Comparison of uncertainty calibration provided by Lemma~\ref{lemma:cheap_UQ}, Lemma~\ref{lemma:expensive_UQ} and Lemma~\ref{lemma:covariance_for_Reid} appears in Section~\ref{section:Numerical_experiments}.

\section{Comparison with \cite{reid2020probabilistic}}
\label{section:Comparison_with_BayesCG}
In recent contribution \cite{reid2020probabilistic}, authors explore related ideas to the construction of probabilistic projection methods. In this section we show that the covariance matrix introduced in \cite[Definition 3.1]{reid2020probabilistic} corresponds to a particular choice of $Y$ and $G$ in Theorem~\ref{theorem:main_result}, we formulate a conjecture about optimality of the low-rank posterior in \cite{reid2020probabilistic}, and comment on uncertainty calibration adopted in \cite{reid2020probabilistic}.

\subsection{Covariance matrix}
\label{subsection:Covariance_matrix}
In \cite{reid2020probabilistic} authors propose to use the following covariance matrix:
\begin{equation}
\label{Reid_prior_cov}
    \Sigma_{0} = \sum_{i=1}^{m+d} \left(\gamma_{i}\left\|r_{i-1}\right\|_2^2\right)\widetilde{v}_{i} \left(\widetilde{v}_{i}\right)^T,~\widetilde{v}_{i}=v_{i}\big/\sqrt{\eta_{i}}
\end{equation}
where $\eta_{i}$, $v_{i}$, $r_{i}$ are as in Algorithm~\ref{algorithm:CG}, and $d\ll n-m$ is a small number of additional CG iterations used to calibrate uncertainty. For this covariance matrix they show that posterior covariance after projection on the first $m$ search directions reads
\begin{equation}
\label{Reid_posterior_cov}
    \Sigma_{m} = \sum_{i=m+1}^{m+d} \left(\gamma_{i}\left\|r_{i-1}\right\|_2^2\right)\widetilde{v}_{i} \left(\widetilde{v}_{i}\right)^T.
\end{equation}

\begin{algorithm}[t]
    \caption{Conjugate gradient.}
    \label{algorithm:CG}
    \begin{algorithmic}[1]
        \STATE \textbf{Input:} positive definite matrix $A$, right-hand side $b$, initial guess $x$, number of sweeps $m$.
        \STATE \textbf{Output:} approximate solution $x$.
        \\\hrulefill
        \STATE $r_{0} = b - Ax$
        \STATE $v_{1} = r_{0}$
        \FOR{$i = 1:m$}
            \STATE $\eta_{i} = v_{i}^{T}Av_{i}$
            \STATE $\gamma_{i} = r_{i-1}^{T}r_{i-1}\big/\eta_{i}$
            \STATE $x_{i} = x_{i-1} + \gamma_{i}v_{i}$
            \STATE $r_{i} = r_{i-1} - \gamma_{i}Av_{i}$
            \STATE $\delta_{i} = r_{i}^Tr_{i}\big/r_{i-1}^{T}r_{i-1}$
            \STATE $v_{i+1} = r_{i} + \delta_{i}v_{i}$
        \ENDFOR
        \STATE $x = x_{m}$
    \end{algorithmic}
\end{algorithm}

We are going to show that covariance matrix \eqref{Reid_prior_cov} is in line with Theorem~\ref{theorem:main_result}. We start with the following supplementary result.
\begin{lemma}
\label{lemma:invariance}
    Mean vector $\widetilde{x}$ in Theorem~\ref{theorem:main_result} does not depend on the choice of bases in subspaces ${\sf Range}\left(V\right)$, ${\sf Range}\left(W\right)$.
\end{lemma}
\begin{proof}
    Let columns of $\widetilde{V}$ and $\widetilde{W}$ be new bases in subspaces ${\sf Range}\left(V\right)$ and ${\sf Range}\left(W\right)$. It is always possible to find invertible square matrices $G_{1}$, $G_{2}$ that perform a change of bases, i.e., $V = \widetilde{V}G_{1}$ and $W = \widetilde{W}G_{1}$. After the substitution of $\widetilde{V}$ and $\widetilde{W}$ in \eqref{projection_method} yields
    \begin{equation*}
        \begin{split}
            \widetilde{x} &= x_0 + \widetilde{V}G_{1} \left(G_{2}^T \widetilde{W}^T A \widetilde{V}G_{1}\right)^{-1}G_{2}^{T}\widetilde{W}^{T}\left(b - A x_{0}\right)\\
            & = x_0 + \widetilde{V} \left(\widetilde{W}^T A \widetilde{V}\right)^{-1}\widetilde{W}^{T}\left(b - A x_{0}\right).
        \end{split}
    \end{equation*}
    So the mean vector does not depend on the choice of basis.
\end{proof}

Now, we show that the following result holds.
\begin{theorem}
\label{theorem:equivalence_Reid}
    For $A>0$ let $Y$, $G$ and $V=W$ be chosen as follows. Columns of $Y\in\mathbb{R}^{n\times(n-m)}$ are search directions $\widetilde{v}_{i} = v_{i}\big/\sqrt{\eta_{i}}$, $i= m+1, \dots, n$, matrix $G\in\mathbb{R}^{(n-m)\times(n-m)}$ is diagonal with elements $G_{ii} = \gamma_{i} \left\|r_{i-1}\right\|_2^2$, $i= m+1, \dots, n$, where $v_{i}$ and $\eta$, $r_{i}$, $\gamma_{i}$ are defined by Algorithm~\ref{algorithm:CG}. Columns of matrix $V$ form a basis for Krylov subspace $\mathcal{K}_{m}\left(A, r_0\right) = {\sf Span}\left\{r_{0}, A r_{0}, \dots, A^{m-1}r_{0}\right\}$.

    Let  the solution to $Ax = b$ be a normal random variable with probability density function $p(x) = \mathcal{N}(x|x_{0}, \Sigma_{0})$, where $\Sigma_{0} = V V^{T} + Y G Y^T$.

    Under this condition the mean of posterior distribution $p(x|W^{T}Ax=W^{T}b)$ coincides with projection method \eqref{projection_method} (and with the output of Algorithm \ref{algorithm:CG} in exact arithmetic), and the covariance matrix is $Y G Y^T = \sum_{i=m+1}^{n} \left(\gamma_{i}\left\|r_{i-1}\right\|_2^2\right)\widetilde{v}_{i} \left(\widetilde{v}_{i}\right)^T$.
\end{theorem}
\begin{proof}
    By construction $\widetilde{v}_{i},~i=1,\dots, m$ form an $A$-orthogonal basis for $\mathcal{K}_{m}\left(A, r_0\right)$. Using Lemma~\ref{lemma:invariance} we can transform matrix $V$, such that columns of new matrix $\widetilde{V}$ are $\widetilde{v}_{i},~i=1,\dots, m$.

    To apply Theorem~\ref{theorem:main_result} we need to check that ${\sf Range}(Y) = {\sf Null}(W^{T}A)$. Indeed, if $y\in {\sf Range}(Y)$ it has a form $y = \sum_{i=m+1}^{n}y_{i}\widetilde{v}_{i}$, we can see that $W^TAy = V^T A y = 0$, because $\widetilde{v}_i^T A \widetilde{v}_{j} = 0$ for $i=1,\dots, m$, $j = m+1,\dots,n$. This means ${\sf Range}(Y)\subset {\sf Null}(W^{T}A)$. Now, if $y\in{\sf Null}(W^{T}A)$ it is $A-$orthogonal to the first $m$ vectors $\widetilde{v}_{i}$, because $\widetilde{v}_{i},~i=1,\dots,n$ form a complete set\footnote{We do not consider ``lucky breakdowns'' \cite[Section 6.3.1]{saad2003iterative}}, we conclude that $y\in{\sf Range}(Y)$ and ${\sf Null}(W^{T}A) \subset {\sf Range}(Y)$.

    Because, ${\sf Range}(Y) = {\sf Null}(W^{T}A)$ we can apply Theorem~\ref{theorem:main_result} which gives us \eqref{projection_method} as mean and $YGY^T$ as a covariance matrix.
\end{proof}
From Theorem~\ref{theorem:equivalence_Reid} we can conclude that the covariance from \cite{reid2020probabilistic} can be considered as a special case of general result given in Theorem~\ref{theorem:main_result}.

Before the comparison on uncertainty calibration we want to discuss a low-rank approximation \eqref{Reid_posterior_cov} to a full-rank matrix $\Sigma_{m}$ from Theorem~\ref{theorem:equivalence_Reid}. Is it the ``best'' low-rank approximation? We believe, that in some sense it is. To motivate this we start with a supplementary statement.
\begin{lemma}
\label{lemma:new_normal}
    For $A>0$ we define the following operator norm $\left\|B\right\|_{A,A^{-1}} \equiv \sup_x \left\|Bx\right\|_A\big/\left\|x\right\|_{A^{-1}}$.
    If $B = \sum_{j=1}^{K} d_{j} u_{j}u_{j}^T$ where $d_{1}\geq d_{2}\geq\dots\geq d_{K}>0, K\leq n$ and $u_{j}^{T}A u_{k} = \delta_{jk}$, the operator norm of $B$ is $\left\|B\right\|_{A, A^{-1}} = d_{1}$.
\end{lemma}
\begin{proof}
    Let columns of $U$ be $u_{j},~j=1,\dots,K$ and $D$ be a diagonal matrix with $D_{jj} = d_{j}$. Using the definition we get
    \begin{equation*}
        \sup_{x} \frac{\sqrt{x^T U D U^T A U D U^T x}}{\sqrt{x^T A^{-1} x}} = \sup_{y} \frac{\sqrt{y^T A U D^2 U^T A y}}{\sqrt{y^T A y}},
    \end{equation*}
    where we used $A-$orthogonality and define $y = A^{-1}x$. Now, without the loss of generality we take $y = U\alpha$ to obtain
    \begin{equation*}
        \left\|B\right\|_{A, A^{-1}} = \sup_{\alpha} \frac{\sqrt{\alpha^T D^2 \alpha}}{\sqrt{\alpha^T \alpha}} = d_{1}.
    \end{equation*}
\end{proof}
Next we extend a well-known optimal low-rank approximation result on norm $\left\|\cdot\right\|_{A, A^{-1}}$.
\begin{lemma}
\label{lemma:Trefethen}
    Let $B$ be the same as in Lemma~\ref{lemma:new_normal}, and $B_{m} = \sum_{j=1}^{m} d_{j} u_{j}u_{j}^T,~m<K$. Then
    \begin{equation*}
        \left\|B - B_{m}\right\|_{A, A^{-1}} = \inf_{{\sf rank} C\leq m}\left\|B - C\right\|_{A, A^{-1}} = d_{m+1}.
    \end{equation*}
\end{lemma}
\begin{proof}
    From Lemma~\ref{lemma:new_normal} we know that $\left\|B - B_{m}\right\|_{A, A^{-1}} = d_{m+1}$. For the second part we use a proof by contradiction from \cite[Theorem 5.8]{MR1444820}.

    Suppose that there is $C,~{\sf rank}(C)\leq m$ for which the norm of the difference is smaller, i.e., $\left\|B - C\right\|_{A, A^{-1}}<d_{m+1}$. Because $C$ has rank $m$ there is a $n-m$ dimensional subspace $R\subset\mathbb{R}^{n}:\forall r\in R \Rightarrow Cr = 0$. This implies
    \begin{equation*}
    \begin{split}
        \left\|Br\right\|_{A, A^{-1}} &= \left\|\left(B-C\right)r\right\|_{A, A^{-1}}\\
        &\leq \left\|B-C\right\|_{A, A^{-1}}\left\|r\right\|_{A^{-1}}<\sigma_{m+1}\left\|r\right\|_{A^{-1}}.
    \end{split}
    \end{equation*}
    We know that in the subspace $\widetilde{R}$ spanned by $u_{j}$, $j=1,\dots,m+1$ the norm of the matrix $B$ fulfills $\left\|B\widetilde{r}\right\|_{A, A^{-1}}\geq d_{m+1} \left\|\widetilde{r}\right\|_{A^{-1}}$. Because for these subspaces $\left|\widetilde{R}\right| + \left|R\right| = n+1$, there is a vector that belongs to both of them. Thus by contradiction $\left\|B - C\right\|_{A, A^{-1}}\geq d_{m+1}$, and the bound is attained by $B_{m}$.
\end{proof}

Lemma~\ref{lemma:Trefethen} implies that approximation \eqref{Reid_posterior_cov} is optimal (best $d-$rank approximation) in $\left\|\cdot\right\|_{A, A^{-1}}$ norm if $\gamma_{i}\left\|r_{i-1}\right\|_2^2, i=1,\dots, n$ form a non-increasing sequence. Unfortunately, this is not the case, because $\left\|r_{i}\right\|$ can increase in the course of iterations.

\begin{figure}
    \centering
    \includegraphics[scale=0.52]{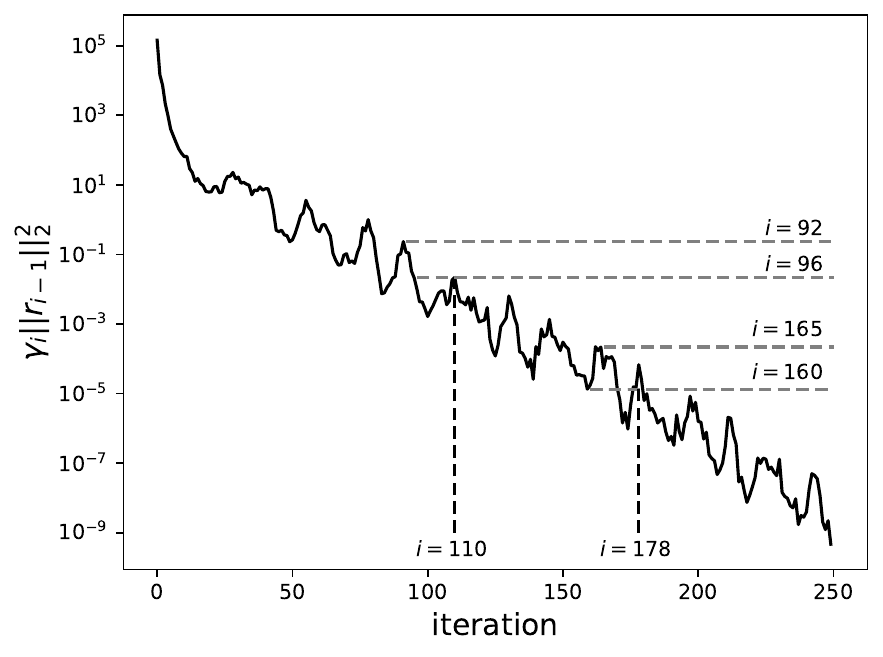}
    \caption{Figure demonstrates $\gamma_{i}\left\|r_{i-1}\right\|_2^2$ for matrix bcsstm07 from SuiteSparse matrix collection.}
    \label{fig:optimality}
\end{figure}

However, because $\left\|r_{i}\right\|\rightarrow 0$ in exact arithmetic, it seems, we still can obtain an optimal low rank approximation for an appropriate choice of $d$ in \eqref{Reid_posterior_cov}. This is exemplified in Figure~\ref{fig:optimality}. Evidently, if $m=91$ for $d\leq4$ we obtain optimal $d-$rank approximation to the whole covariance matrix $\Sigma_{m}$ from Theorem~\ref{theorem:equivalence_Reid}. However if we take $4<d<8$ we achieve no improvement over $d=4$ because the next peak $i=110$ has larger $\gamma_{i}\left\|r_{i-1}\right\|_2^2$. Less favourable situation occurs when $m=159$. In this case all $d<5$ does not result in optimal $d-$rank approximation, and $d=5$ gives an optimal $1-$rank approximation. Based on these observations we formulate the following conjecture.

\begin{conjecture}
\label{conjecture:low_rank_approximation}
    For almost any positive definite matrix $A\in\mathbb{R}^{n\times n}$, for any iteration $m\leq n$, there is a $d(m)\ll n$ and $r(d)\leq d(m)$ such that a covariance matrix $\Sigma_{m} = \sum_{i=m+1}^{m+d(m)} \left(\gamma_{i}\left\|r_{i-1}\right\|_2^2\right)\widetilde{v}_{i} \left(\widetilde{v}_{i}\right)^T$ is an optimal $r(d)-$rank approximation to the full covariance matrix $\widetilde{\Sigma}_{m} = \sum_{i=m+1}^{n} \left(\gamma_{i}\left\|r_{i-1}\right\|_2^2\right)\widetilde{v}_{i} \left(\widetilde{v}_{i}\right)^T$ with respect to the operator norm $\left\|\cdot\right\|_{A, A^{-1}}$.
\end{conjecture}

Note, that the conjecture, if correct, merely ensures an optimality of approximation \eqref{Reid_posterior_cov} to the full covariance matrix. Conjecture~\ref{conjecture:low_rank_approximation} does not tell whether the full matrix is optimal for uncertainty quantification in some (yet undefined) sense.

\subsection{Comparison of uncertainty calibration}
\label{subsection:Uncertainty_calibration}
Unlike previous works \cite{bartels2019probabilistic}, \cite{cockayne2018bayesian} in article \cite{reid2020probabilistic} authors focus on $A-$norm of error. For this choice it is easy to construct an underestimate for an error $\left\|x^{\star} - x_{m}\right\|_{A}^2$ using information, available as a byproduct of Algorithm~\ref{algorithm:CG}. Namely, this is done by the following expression \cite[4.1]{reid2020probabilistic}, \cite[Theorem 5:3]{hestenes1952methods}
\begin{equation}
    \left\|x^{\star} - x_{m}\right\|_A^{2} - \left\|x^{\star} - x_{m+d}\right\|_A^{2} = \sum_{m+1}^{m+d}\gamma_{i}\left\|r_{i-1}\right\|_{2}^{2},
\end{equation}
from which we conclude that
\begin{equation}
\label{underestimation}
     \left\|x^{\star} - x_{m}\right\|_A^{2} \geq \sum_{m+1}^{m+d}\gamma_{i}\left\|r_{i-1}\right\|_{2}^{2}.
\end{equation}
The advantage of a posterior covariance matrix defined by \eqref{Reid_posterior_cov} is that to compute it one needs to perform a few additional iterations of conjugate gradient and store $A-$orthogonal directions $v_i$ and scales $\gamma_i \left\|r_{i-1}\right\|_2^2$. So the estimation of $\Sigma_{m}$ is cheap and justified by \eqref{underestimation}.
However, in our opinion there are several disadvantages. First, even when $\left\|e_{i}\right\|_{A}$ is small $\left\|e_{i}\right\|$ can remain large in the subspace corresponding to small eigenvalues of $A$. Second, \eqref{underestimation} provides only underestimate, which can be misleading in case of slow convergence (see Figure~\ref{fig:UQ_Reid} for an example of this behavior for biharmonic equation).

Our approach to uncertainty calibration is based on Lemma~\ref{lemma:covariance_for_Reid} and Algorithm~\ref{algorithm:UQ_calibration} with ${\sf statistic} = S$. Algorithm~\ref{algorithm:UQ_calibration} simply perform an additional run of a projection method (conjugate gradient in this case) for a known $x^{\star}$, and records $\left\|x^{\star}-x_{m}\right\|_{A}^{2}$. This norm is then used as an estimation for an error with a target right-hand side $b$ for which $x^{\star}$ is unknown. We will see that this approach leads to more reasonable $S-$statistic. The obvious disadvantage is a much higher cost of uncertainty calibration. However, our approach can be cheaper in case one needs to solve a set of linear equation with different right-hand sides and the same matrix $A$ (Section~\ref{section:Numerical_experiments} contain a relevant example).

\section{Numerical experiments}
\label{section:Numerical_experiments}
Julia \cite{MR3605826} code that reproduces experiments in this section is available at \url{https://github.com/VLSF/BayesKrylov}.

\subsection{Comparison with \cite{bartels2019probabilistic}}
\label{subsection:Comparison_with_Probabilistic_Projection_Methods}

\begin{figure*}
    \centering
    \includegraphics[scale=0.4]{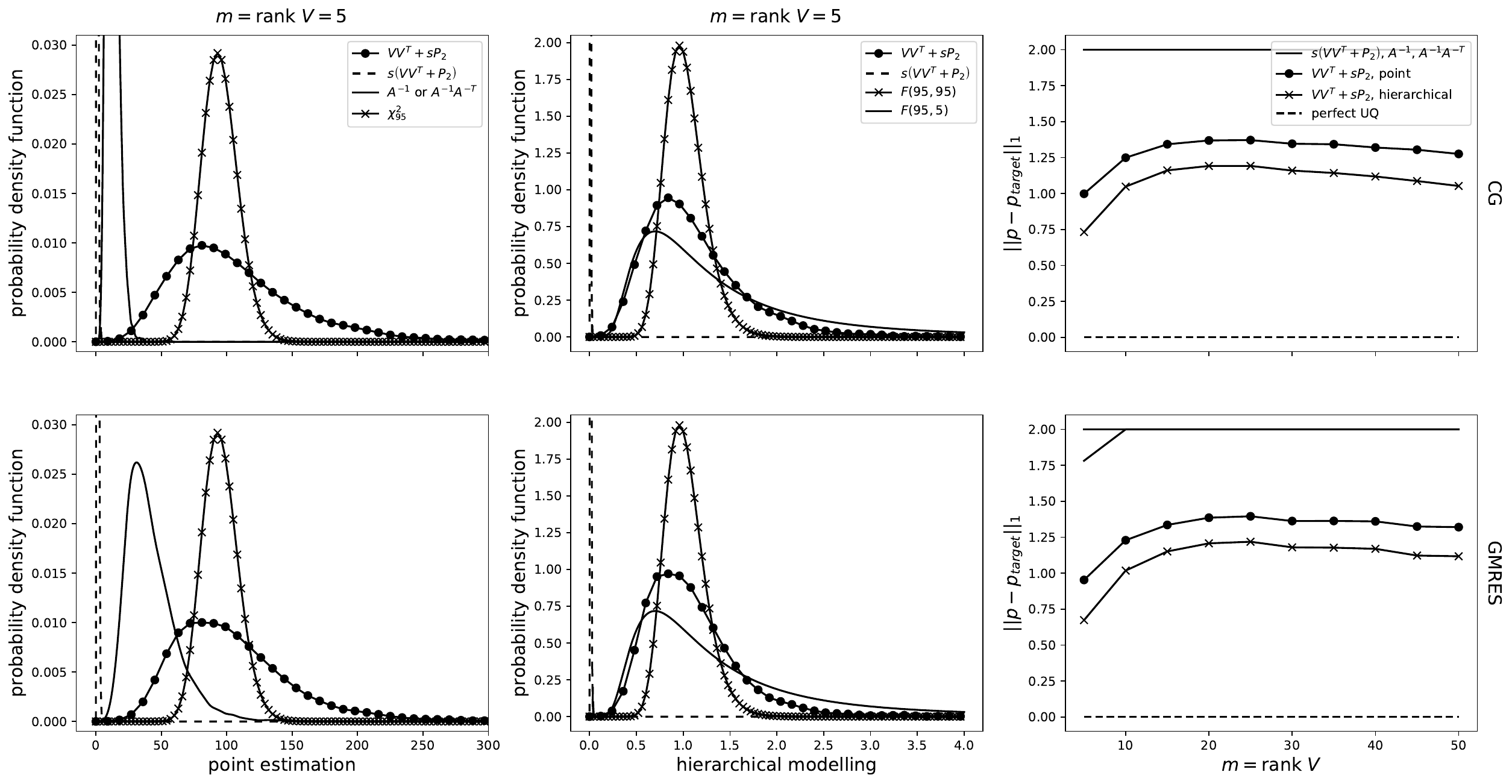}
    \caption{Figures demonstrate theoretical test statistics and empirical distributions for different prior distributions. Common legends for each column appear in the first row. The legend provides specifications of covariance matrices. For example, $s\left(VV^{T} + P_2\right)$ refers to posterior described in Lemma~\ref{lemma:cheap_UQ} with $\Psi = P_2$. The first two columns contain point estimation and hierarchical modelling for five projection steps. The first row presents results related to the conjugate gradient method and the second to GMRES. In the last column we show how $L_1$ norm of the difference between empirical $p_{e}$ and target $p_{t}$ ($\chi^{2}$ or $F$ as explained in Section~\ref{section:Numerical_experiments}) distributions changes with the number of projection steps. Perfect uncertainty calibration corresponds to zero value of discrepancy. The worst possible mismatch corresponds to $L_1$ norm of the error equals two. Overall we can see that the method proposed in Lemma~\ref{lemma:expensive_UQ} provides a reasonable uncertainty for both projection processes.}
    \label{fig:UQ}
\end{figure*}

To assess the uncertainty calibration, we compare theoretical distributions for test statistics with empirical probability density functions averaged over many matrices. Note, that unlike $S-$statistic, $Z-$statistic for perfectly calibrated uncertainty does not depend on the matrix for both point estimation and hierarchical modelling. This makes averaging over $A$ legitimate. Details of this procedure are summarized in Algorithm~\ref{algorithm:UQ_assessment}.
\begin{algorithm}
    \caption{UQ assessment.}
    \label{algorithm:UQ_assessment}
    \begin{algorithmic}[1]
        \STATE \textbf{Input:} Distributions for matrix $p(A)$; exact solution $p(x^{\star})$; number of search directions $m$; projection method $V, W \leftarrow {\sf Proj}\left(A, b, m\right)$,; number of samples $N$; statistics $p(z)\leftarrow {\sf Stat}(e_1, \dots, e_N)$; number of matrices $M$.
        \STATE \textbf{Output:} test statistic.
        \\\hrulefill
        \FOR{$i = \overline{1, M}$}
            \STATE $A_{i}\sim p(A)$
            \FOR{$j = \overline{1, N}$}
                \STATE $x_{j}^{\star} \sim p(x^{\star})$
                \STATE $b_{ij} = A_{i} x^{\star}_{j}$
                \STATE $V_{ij}, W_{ij} \leftarrow  {\sf Proj}\left(A_{i}, b_{ij}, m\right)$
                \STATE $\widetilde{x}_{ij} = V_{ij} \left(W_{ij}^T A_{i} V_{ij}\right)^{-1}W_{ij}^{T}b_{ij}$
                \STATE $e_{ij}\leftarrow \widetilde{x}_{ij} - x^{\star}_{j}$
            \ENDFOR
        \ENDFOR
        \STATE $p(z)\leftarrow {\sf Stat}(e_{11}, \dots, e_{NN})$
    \end{algorithmic}
\end{algorithm}

Details on components of Algorithm~\ref{algorithm:UQ_assessment} are as follows:
\begin{itemize}
    \item[$p(A)$:] To draw symmetric positive definite matrices $A = U D U^{T}$ we sample stacked eigenvectors $U$ from uniform distribution over $O(n)$, and eigenvalues from exponential distribution with scale $\widetilde{s}$.
    \item[$p(x^{\star})$:] As a distribution of exact solution we take standard multivariate normal $\mathcal{N}\left(\cdot|0, I\right)$ as in \cite{cockayne2018bayesian}.
    \item[${\sf Proj}:$] Two projection processes are used. The first one with $W = V = \begin{pmatrix}\widetilde{b} & A\widetilde{b} & \cdots & A^{m-1}\widetilde{b}\end{pmatrix}$, $\widetilde{b} = b\big/\left\|b\right\|_2$ is equivalent to conjugate gradient in exact arithmetic. The second one with $W = AV$, and $V=\begin{pmatrix}\widetilde{b} & A\widetilde{b} & \cdots & A^{m-1}\widetilde{b}\end{pmatrix}$, $\widetilde{b} = b\big/\left\|b\right\|_2$ is equivalent to GMRES under the same condition.
    \item[${\sf Stat}:$] For distribution $p(x) = \mathcal{N}(x|\mu, \Sigma)$, ${\sf rank}(\Sigma)= n - m$ test statistic is $z = \left(x-\mu\right)^{T}\Sigma^{\dagger}\left(x-\mu\right)\sim \chi^{2}_{n-m}$, and for multivariate Student distribution ${\sf St}_{\nu}\left(\mu, \Sigma\right)$, test statistic is $z=\left(x-\mu\right)^{T}\Sigma^{\dagger}\left(x-\mu\right)/(n-m) \sim F(n-m, \nu)$.
\end{itemize}

\begin{figure*}
    \centering
    \includegraphics[scale=0.4]{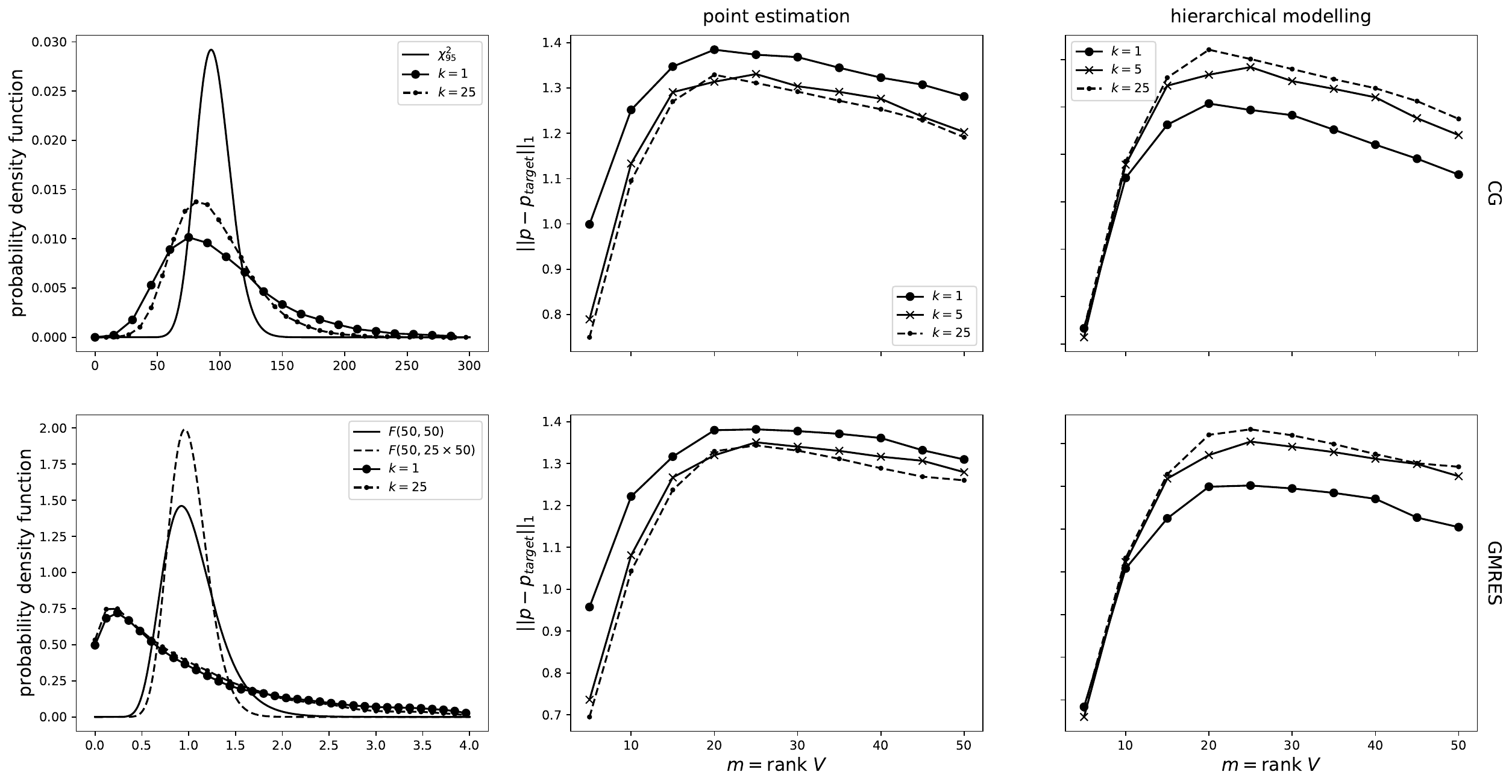}
    \caption{Figures summarize the dependence of proposed uncertainty calibration (Algorithm~\ref{algorithm:UQ_calibration}) on the number of additional observations $k$. First row corresponds to results for conjugate gradient iteration and the second row -- for GMRES iteration. The second and the third columns, which contain point estimation and hierarchical modelling, respectively, share common legends that appeared in the first row. Graphs in these last two columns show how $L_1$ norm of the difference between empirical $p_{e}$ and target $p_{t}$ ($\chi^{2}$ or $F$ as explained in Section~\ref{section:Numerical_experiments}) distributions changes with the number of projection steps for $k = 1, 5, 25$ additional observations in Algorithm~\ref{algorithm:UQ_calibration}. Figures in the first column allow for visual inspection of empirical and target distributions for $Z$-statistic. Namely, for CG, we sketch the probability density function of $Z$-statistic for point estimation in the first row (the target distribution is $\chi^2$), whereas the second row contains the same quantity but for hierarchical modelling (the target distribution is $F$). We can see that for point estimation, additional observations marginally improve uncertainty calibration, whereas, for hierarchical modelling, the situation is reversed. We conclude that, first, it makes little sense to use $k>1$ for the chosen family of linear systems. Second, such behaviour clearly indicates that the chosen statistical model is inadequate for Krylov subspace methods.}
    \label{fig:UQ_obs}
\end{figure*}

In all experiments, the size of the problem is $n = 100$, the scale is $\widetilde{s} = 10$, number of matrices $M = 500$, number of samples is $N = 20$. We also take $G = I$, $\alpha=\beta=0$\footnote{Note that the choice $\alpha = \beta = 0$ leads to the improper prior. In the present case the posterior distribution is always proper, so noninformative prior seems harmless.  Moreover, $s$ is a scale parameter so $p(s)\propto s^{-1}$ is a reasonable choice (see \cite[Section 2.8]{gelman2013bayesian}).} in both Lemma~\ref{lemma:cheap_UQ} and Lemma~\ref{lemma:expensive_UQ}, and use Algorithm~\ref{algorithm:UQ_calibration} with ${\sf statistic} = Z$ and $k=1$, i.e., a single additional sample, to calibrate uncertainty. Results of  Lemma~\ref{lemma:cheap_UQ} and Lemma~\ref{lemma:expensive_UQ} are used in two regimes. The first one is point estimation. In this case parameters $\widetilde{\alpha}$, $\widetilde{\beta}$ of inverse-gamma distribution are used to find a mean value $\mathbb{E}[s] = \beta\big/(\alpha-1)$, and this mean value is used as a scale in covariance matrix $sP_2$. As a result, the statistic $Z(x)$ is compared with $\sim \chi_{n-m}^2$. The second one is a hierarchical modelling. In this case $s$ is marginalized (as in second parts of Lemma~\ref{lemma:cheap_UQ} and Lemma~\ref{lemma:expensive_UQ}) and the resulting statistic $Z(x)$ is compared with $F(n-m, 2\widetilde{\alpha})$. More precisely, according to Lemma~\ref{lemma:cheap_UQ} for prior with covariance matrix $s(VV^T + P_2)$ and no additional observations the target distribution is $F(n-m, 2\alpha + m) = F(n-m, m)$, whereas Lemma~\ref{lemma:expensive_UQ} implies that for prior with covariance matrix $VV^T + sP_2$ and $k$ additional observations (see Algorithm~\ref{algorithm:UQ_calibration}) we should use $F(n-m, 2\alpha + k(n-m)) = F(n-m, k(n-m))$ as a target distribution.

As a distance between distributions we choose standard $L_1$ norm $d(p_1, p_2) \equiv \int dx\left|p_1(x) - p_2(x)\right|$ approximated by central Riemann sum. Probability density is computed with RBF kernel density estimator.

The results are presented in Figure~\ref{fig:UQ} ($k=1$ in Algorithm~\ref{algorithm:UQ_calibration}) and Figure~\ref{fig:UQ_obs} ($k=1, 5, 25$ in Algorithm~\ref{algorithm:UQ_calibration}). From data presented on Figure~\ref{fig:UQ} it follows, that covariance matrices $A^{-1}$, $A^{-1}A^{-T}$ and $s(VV^{T} + P_2)$ (Lemma~\ref{lemma:cheap_UQ} with $\Psi = P_2$) fail to provide meaningful uncertainty calibration. The only reasonably tuned variant is given by covariance $VV^{T} + sP_2$ (Lemma~\ref{lemma:expensive_UQ}), where $s$ is fixed with additional observation $Px$. We can also see that the hierarchical modelling is marginally better than the point estimation. Figure~\ref{fig:UQ_obs} describes how uncertainty calibration depends on the number of observations $k$. We can see that when $k$ increases, the calibration for point estimation slightly improves, whereas the increase in $k$ leads to the degradation of uncertainty calibration for the hierarchical modelling. Nowhere the convergence to theoretical distribution is observed when $k$ is increased. This pathological behaviour supports the discussion in Section~\ref{section:When_probabilistic_projection_methods_are_sound}, where we state that probabilistic projection methods in they current form are unsuitable for Krylov subspace methods.

\begin{figure*}
    \centering
    \includegraphics[scale=0.4]{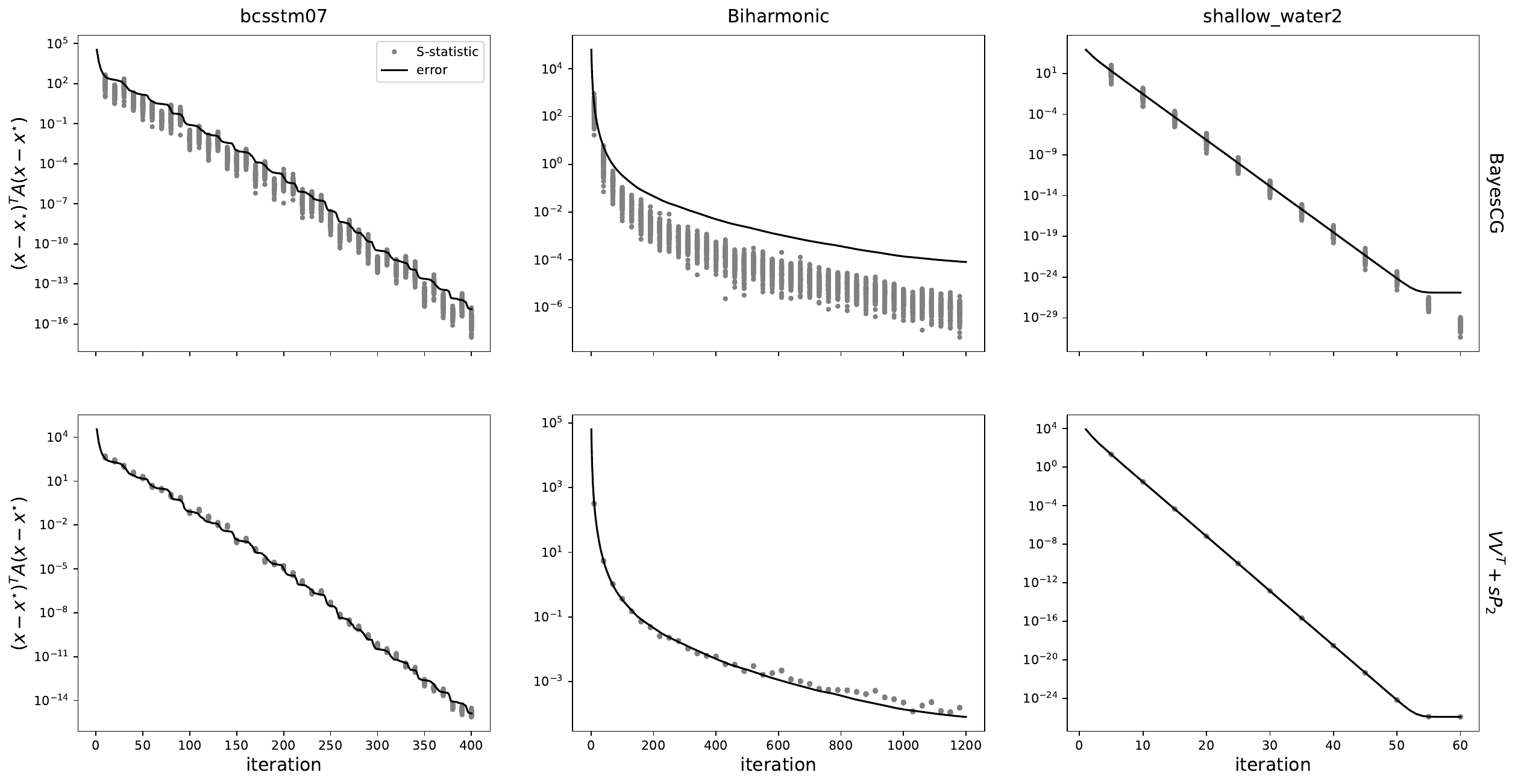}
    \caption{Figures demonstrate exact error $e_{m}^TAe_{m}$ on iteration $m$, and samples from $S-$statistic for three matrices. First row corresponds to uncertainty calibration proposed in \cite{reid2020probabilistic}. Second row shows samples from $S-$statistic calibrated according to Algorithm~\ref{algorithm:UQ_calibration} with ${\sf statistic} = S$. We can see that the statistical uncertainty calibration proposed in this article leads to better uncertainty in all three cases.}
    \label{fig:UQ_Reid}
\end{figure*}

\subsection{Comparison with \cite{reid2020probabilistic}}
\label{subsection:Comparison_with_BayesCG}

In this case, we use Algorithm~\ref{algorithm:UQ_calibration} with ${\sf statistic} = S$ and $k=1$. Note, that because for large $m$ the effect of rounding errors is significant, we use conjugate gradient to compute projection operator $P_1$. If one computes $P_1$ as in Algorithm~\ref{algorithm:UQ_calibration}, it gives an underestimation of error for large $m$, because in this case methods based on projection method \eqref{projection_method} converge much faster than the conjugate gradient as defined in Algorithm~\ref{algorithm:CG}.

For a given matrix $A>0$ we compare uncertainty calibration as follows. For method described in \cite{reid2020probabilistic} we sample $\delta$ from $\mathcal{N}(\cdot|0, \Sigma_{m})$, where $\Sigma_{m}$ is a posterior covariance matrix \eqref{Reid_posterior_cov} and plot $l=100$ samples from $S-$statistic $\delta^T A \delta$ for $m$ in regular intervals (each $10$ or each $20$ iterations). For our approach we use Algorithm~\ref{algorithm:UQ_calibration} with ${\sf statistic} = S$ and $k=1$, take $\mathbb{E}[s] = \widetilde{\beta}\big/(\widetilde{\alpha} - 1)$ and sample from $\mathbb{E}[s]\chi_{n-m}^2$, which is equivalent to $S-$statistic as explained in Lemma~\ref{lemma:simplified_S_statistic}. Results for test problems can be found in Figure~\ref{fig:UQ_Reid}. Overall, we can see that our approach leads to much better uncertainty calibration in all cases. The price for it is much more expensive uncertainty calibration than the one adopted in \cite{reid2020probabilistic}. Results for individual matrices are discussed below.

We use three positive definite matrices $A$:
\subsubsection{bcsstm07}
The first example is a symmetric positive definite $n=420$ matrix from SuiteSparse Matrix Collection: \url{https://sparse.tamu.edu/HB/bcsstm07}.

From the first column of Figure~\ref{fig:UQ_Reid} we can conclude that the method of \cite{reid2020probabilistic} leads to underestimation for approximately an order of magnitude for each iteration. Our approach gives almost exact error estimation in this case.

\subsubsection{Biharmonic equation}
For the second test problem we take biharmonic equation
\begin{equation}
\label{Biharmonic_equation}
    \begin{split}
    \frac{\partial^4}{\partial x^4}u(x, y) + 2\frac{\partial^4}{\partial x^2\partial y^2}u(x, y) + \frac{\partial^4}{\partial y^4}u(x, y) = f(x, y),\\~x, y \in \left[0, 1\right]^2,~\left.u(x, y)\right|_{\partial\Gamma} = 0,~\left.\partial_{n} u(x, y)\right|_{\partial\Gamma} = 0,
    \end{split}
\end{equation}
here $\partial_{n}$ is a derivative along the normal direction to the boundary $\partial \Gamma$. To discretize this equation, we use centered second-order finite difference approximation given by a $13$ point stencil
\begin{equation}
    s =
    \left[\begin{matrix}
    &&1&&\\
    &2&-8&2&\\
    1&-8&20&-8&1\\
    &2&-8&2&\\
    &&1&&
    \end{matrix}\right],
\end{equation}
with appropriate modification near the boundary (see \cite[Section 4]{tong1992multilevel}). Along each direction we take $n_x = n_y = (2^7 - 1)$ which results in size $n = 16129$ positive definite matrix.

Results for this equation are in the second column of Figure~\ref{fig:UQ_Reid}. The condition number is large and the convergence is extremely slow. As a result, uncertainty calibration from \cite{reid2020probabilistic} is poor. For example at $m=1200$ the exact error norm is about $\simeq 10^{-3}$, whereas an estimation is $\simeq10^{-6}$. Our statistical uncertainty calibration results in a mild overestimation of the exact error, which is better than the uncertainty from \cite{reid2020probabilistic}.

\subsubsection{shallow\_water2}

The third example is symmetric positive definite $n=81920$ matrix from SuiteSparse matrix collection: \url{https://sparse.tamu.edu/MaxPlanck/shallow_water2}.

Last column of Figure~\ref{fig:UQ_Reid} provides a summary of results. The convergence is good and for all practical purposes both our approach and the method from \cite{reid2020probabilistic} provide a reasonable estimation of error. The only difference is that our approach leads to smaller variance of the test statistic.

\begin{figure*}
    \centering
    \includegraphics[scale=0.4]{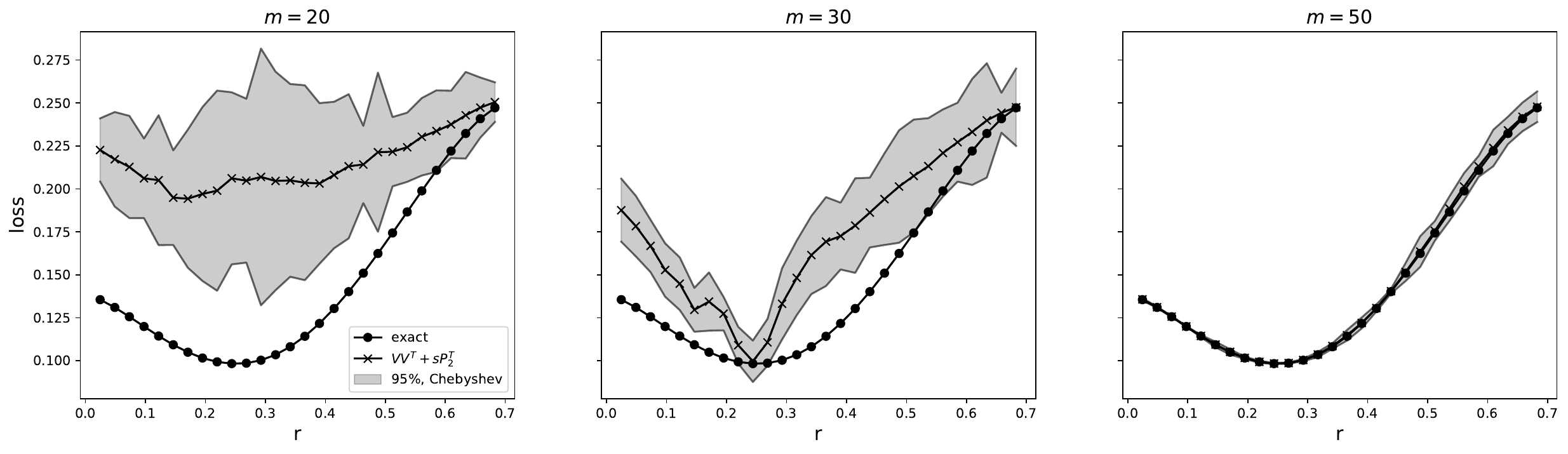}
    \caption{Figures demonstrate comparison of exact loss function \eqref{PDE_constrained_discrete} with an estimation obtained from probabilistic projection method from Theorem~\ref{theorem:orthogonal_projector} with $W = V = \begin{pmatrix}\widetilde{b} & A\widetilde{b} & \cdots & A^{m-1}\widetilde{b}\end{pmatrix}$, $\widetilde{b} = b\big/\left\|b\right\|_2$ for $m=20,~30,~50$. Shaded region is enclosed by curves $\mu_{m}(r)\pm \sigma_m(r)$, where $\mu_{m}(r)$ is an approximate mean value of $\mathcal{L}(r)$ and $\sigma_m(r)^2$ is approximate variance, both estimated using $30$ samples from the posterior distribution specified in Theorem~\ref{theorem:orthogonal_projector}.}
    \label{fig:UQ_PDE}
\end{figure*}

\subsection{Uncertainty quantification for PDE-constrained optimization}
\label{subsection:PDE_constrained}

As a last example we consider an optimal heating problem. Consider a diffusive heat transfer \cite[Section 5.1.3]{pletcher2012computational} from four point heat sources with unit heat fluxes in simple geometry
\begin{equation}
\label{heat_equation}
    \begin{split}
        &-\frac{\partial^2 T(x, y)}{\partial x^2}-\frac{\partial^2 T(x, y)}{\partial y^2} = \sum_{i=1}^{4} \delta(x - x_{i}) \delta (y - y_{i}), \\
        &x, y \in \Gamma \equiv\left[0, 1\right]^2, \left.T(x, y)\right|_{\partial \Gamma} = 0,
    \end{split}
\end{equation}
where $x_i, y_i$ are located in vertices of the square:
\begin{equation}
\begin{split}
    x_1 &= r\cos\left(\pi\big/4\right),~y_1 = r\sin\left(\pi\big/4\right);\\
    x_2 &= -r\cos\left(\pi\big/4\right),~y_2 = r\sin\left(\pi\big/4\right);\\
    x_3 &= -r\cos\left(\pi\big/4\right),~y_3 = -r\sin\left(\pi\big/4\right);\\
    x_4 &= r\cos\left(\pi\big/4\right),~y_4 = -r\sin\left(\pi\big/4\right).\\
\end{split}
\end{equation}
We consider the following PDE-constrained optimization problem
\begin{equation}
\label{PDE_constrained}
    \min_{r} \int dx dy \left(T(x, y) - T_{{\sf target}}\right)^2 \text{ s.t. } T(x, y) \text{ solves \eqref{heat_equation}}.
\end{equation}
Physically, the solution to the problem \eqref{PDE_constrained} is a distribution of sources that results in a smallest deviation of temperature field from the target temperature.

We use the finite element method (see \cite{ciarlet2002finite} for introduction) to discretise equation \eqref{heat_equation}. Namely, we approximate temperature field by finite sum
\begin{equation}
    \widetilde{T}(x, y) = \sum_{i=1}^{2^L-1}\sum_{j=1}^{2^L-1}\widetilde{T}_{ij} \phi^{L}_{i}(x)\phi^{L}_{j}(y),
\end{equation}
where $\phi_{i}^{L}(x) = \phi^{L}(x-x_{i}),i=1,\dots,2^{L}-1$ are rescaled and translated copies of a tent function
\begin{equation}
    \begin{split}
    \phi^{L}(x) = \left(1 + x\big/2^{L}\right)\text{Ind}\left[-1\big/2^{L}\leq x\leq0\right]+\\
    +\left(1 - x\big/2^{L}\right)\text{Ind}\left[0<x\leq1\big/2^{L}\right].
    \end{split}
\end{equation}
We then enforce the PDE in a weak form \eqref{heat_equation}, i.e., we apply the same Petrov--Galerkin condition that is in use for projection methods
\begin{equation}
\label{weak_form}
    \begin{split}
        \int_{\Gamma} dxdy~ \phi^{L}_{i}(x)\phi^{L}_{j}(y)\Bigg(-\frac{\partial^2 \widetilde{T}(x, y)}{\partial x^2}-\frac{\partial^2 \widetilde{T}(x, y)}{\partial y^2}\Bigg.\\\Bigg. - \sum_{i=1}^{4} \delta(x - x_{i}) \delta (y - y_{i})\Bigg) = 0.
    \end{split}
\end{equation}
Weak form \eqref{weak_form} leads to the system of linear equations\footnote{This equation can be rearranged into an ordinary linear system $Ax = b$, where $A$ is a matrix with two indices, by the use of lexicographic order. We do not cover this here in details, consult \url{https://github.com/VLSF/BayesKrylov} for the implementation.} $\sum_{k, l}A_{ikjl}\widetilde{T}_{kl} = b_{ij}$ that approximate continuous problem \eqref{heat_equation}. As a discrete counterpart of the continuous PDE-constrained optimization problem \eqref{PDE_constrained} we use the following
\begin{equation}
\label{PDE_constrained_discrete}
    \begin{split}
    &\mathcal{L}(r) \equiv \frac{1}{\left(2^{L}-1\right)^2}\sum_{i,j=1}^{2^{L}-1}\left(\widetilde{T}_{ij}(r) - T_{{\sf target}}\right)^2,\\
    &\min_{r} \mathcal{L}(r) \text{ s.t. } \sum_{k, l = 1}^{2^{L}-1}A_{ikjl}\widetilde{T}_{kl} = b_{ij}(r).
    \end{split}
\end{equation}
To test the uncertainty calibration, we approximate a solution of linear system using probabilistic projection method with $W = V = \begin{pmatrix}\widetilde{b} & A\widetilde{b} & \cdots & A^{m-1}\widetilde{b}\end{pmatrix}$, $\widetilde{b} = b\big/\left\|b\right\|_2$ and sample $\widetilde{T}$ from the posterior distribution. This procedure turns loss function $\mathcal{L}(r)$ into a random variable.

The resulting uncertainty and the loss function are depicted in Figure~\ref{fig:UQ_PDE}. We take $L=6$, so the size of the matrix is $n=3969$, $T_{\sf target} = 0.5$, and access three approximate solutions using ${\sf rank}(V) \equiv m = 20,~30,~50$. In each case we retrieve $30$ samples from $\mathcal{L}(r)$ and estimate mean $\mu_{m}(r)$ and variance $\sigma_{m}^2(r)$. The shaded region in Figure~\ref{fig:UQ_PDE} lies in-between curves $\mu_{m}(r)\pm 5\sigma_{m}(r)$. According to the Chebyshev inequality it contains a given sample from $\mathcal{L}(r)$ with probability $0.96$. In addition to $\mu_{m}(r)$ and variance $\sigma_{m}^2(r)$, Figure~\ref{fig:UQ_PDE} contains an ``exact'' loss function obtained from \eqref{PDE_constrained_discrete}, where linear system is solved with LU decomposition. Note, that since for all $r$ the same linear system is solved, we perform the uncertainty calibration (using $\mathbb{E}[s]$ from Lemma~\ref{lemma:expensive_UQ}) only once. So, the present example demonstrates that our uncertainty calibration can be cheaper than the one, proposed in \cite{reid2020probabilistic}.

From Figure~\ref{fig:UQ_PDE} we can see that the uncertainty calibration is not ideal. For example, in the case $m=30$ the exact value of $\mathcal{L}(r)$ is confidently rejected for $r\leq 0.2$ and $0.3\leq r<0.5$, the same is true for $m=20$ for $r\leq0.5$. Despite this fact, we argue that the present uncertainty is useful. Observe, that for $m=20$ the largest value $\sigma_{20}(r)$ resides in the region that corresponds to the smallest value of the exact loss. This fact can be exploited as follows. A natural way to perform a PDE-constrained optimization is to fit a surrogate model \cite[Section 5]{peherstorfer2018survey}, using multifidelity Gaussian process (see \cite{kennedy2000predicting} for a well-known example of a multifidelity model). The most widely used exploration rules (see \cite[Section IV]{shahriari2015taking}) are directly related to the variance $\sigma(r)$, which contains $\sigma_{m}$. To exemplify, the well known principle coined ``optimism in the face of uncertainty'' (see \cite[Section 7.1]{lattimore2020bandit}) used in the construction of UCB exploration rules, prescribes to choose the next point according to $\arg\min\left(\mu_{m}(r) - \sigma_{m}(r)\right)$. As such, with the present uncertainty calibration Gaussian process favours a correct region for the further exploration.
\section{Conclusion}
\label{section:Conclusion}
In the present work, we solved a problem of vanishing posterior covariance matrix from \cite{bartels2019probabilistic}. Our prior distribution allows for reconstructing the arbitrary projection method and results in a useful computationally inexpensive covariance matrix. We demonstrate on a set of linear problems that our statistical uncertainty calibration matches or outperforms the other existing approaches. As an application we consider a PDE-constrained optimization problem, for which we find that uncertainty is reasonable, albeit is not ideal.

We would like to stress that currently no probabilistic projection method (including the one developed in the current contribution) can rigorously reconstruct realistic Krylov subspace methods. However, a Bayesian interpretation of a two-grid AMG operator is possible. Since uncertainty is perfectly calibrated for AMG, it should be possible to exploit the proposed covariance matrix to construct an optimal projection operator.

\bibliographystyle{apalike}
\bibliography{refs.bib}
\end{document}

%% file: directional_UQ.tex
\begin{tikzpicture}
    \draw[ultra thick, -Latex] (0, 0) -- (6, 3) node[midway, above]{\LARGE $v_{1}$};
    \draw[ultra thick, -Latex] (0, 0) -- (6, -3) node[midway, below]{\LARGE $v_{2}$};
    \draw[ultra thick, -Latex] (0, 0) -- (0, 5) node[midway, left]{\LARGE $v_{3}$};
    \draw[ultra thick, -Latex] (0, 5) -- (0, 10) node[midway, left]{\LARGE $2v_{3}$};
    \draw[ultra thick] (0, 0) -- (0, 10);
    \draw[thick, dashed] (6, 3) -- (12, 0);
    \draw[thick] (6, -3) -- (12, 0);
    \draw[thick, dashed] (0, 5) -- (6, 8) -- (12, 5);
    \draw[thick] (12, 5) -- (6, 2) -- (0, 5);
    \draw[thick] (0, 10) -- (6, 13) -- (12, 10) -- (6, 7) -- (0, 10);
    \draw[thick, dashed] (6, 13) -- (6, 3);
    \draw[thick] (12, 10) -- (12, 0);
    \draw[thick] (6, 7) -- (6, -3);
    \draw[ultra thick, -Latex] (0, 0) -- (12, 5) node[right]{\LARGE $u_{1}$};
    \draw[ultra thick, -Latex] (0, 0) -- (12, 10) node[right]{\LARGE $u_{2}$};
    \draw[ultra thick, -Latex] node[below, xshift=9cm]{\LARGE $P_{\perp}u_{i},i=1,2$} (0, 0) -- (12, 0);
    \draw [ultra thick, -] (0:2.1)  arc (0:22.5:2.1) node [right,pos=0.5] {\LARGE  $\theta_1$};
    \draw [ultra thick, -] (0:5.1)  arc (0:39.8:5.1) node [right,pos=0.5] {\LARGE  $\theta_2$};
\end{tikzpicture}

%% file: UncertaintyCalibration.bbl
\begin{thebibliography}{}

\bibitem[Bartels et~al., 2019]{bartels2019probabilistic}
Bartels, S., Cockayne, J., Ipsen, I. C.~F., and Hennig, P. (2019).
\newblock Probabilistic linear solvers: a unifying view.
\newblock {\em Stat. Comput.}, 29(6):1249--1263.

\bibitem[Bernardo and Smith, 2009]{bernardo2009bayesian}
Bernardo, J.~M. and Smith, A.~F. (2009).
\newblock {\em Bayesian theory}, volume 405.
\newblock John Wiley \& Sons.

\bibitem[Bezanson et~al., 2017]{MR3605826}
Bezanson, J., Edelman, A., Karpinski, S., and Shah, V.~B. (2017).
\newblock Julia: a fresh approach to numerical computing.
\newblock {\em SIAM Rev.}, 59(1):65--98.

\bibitem[Ciarlet, 2002]{ciarlet2002finite}
Ciarlet, P.~G. (2002).
\newblock {\em The finite element method for elliptic problems}.
\newblock SIAM.

\bibitem[Cockayne et~al., 2019]{cockayne2018bayesian}
Cockayne, J., Oates, C.~J., Ipsen, I. C.~F., and Girolami, M. (2019).
\newblock A {B}ayesian conjugate gradient method (with discussion).
\newblock {\em Bayesian Anal.}, 14(3):937--1012.

\bibitem[Gelman et~al., 2013]{gelman2013bayesian}
Gelman, A., Carlin, J.~B., Stern, H.~S., Dunson, D.~B., Vehtari, A., and Rubin,
  D.~B. (2013).
\newblock {\em Bayesian data analysis}.
\newblock CRC press.

\bibitem[Hackbusch, 2016]{MR3495481}
Hackbusch, W. (2016).
\newblock {\em Iterative solution of large sparse systems of equations},
  volume~95 of {\em Applied Mathematical Sciences}.
\newblock Springer, [Cham], second edition.

\bibitem[Hennig, 2015]{hennig2015probabilistic}
Hennig, P. (2015).
\newblock Probabilistic interpretation of linear solvers.
\newblock {\em SIAM J. Optim.}, 25(1):234--260.

\bibitem[Hennig and Kiefel, 2013]{hennig2013quasi}
Hennig, P. and Kiefel, M. (2013).
\newblock Quasi-{N}ewton methods: a new direction.
\newblock {\em J. Mach. Learn. Res.}, 14:843--865.

\bibitem[Hestenes et~al., 1952]{hestenes1952methods}
Hestenes, M.~R., Stiefel, E., et~al. (1952).
\newblock {\em Methods of conjugate gradients for solving linear systems},
  volume~49.
\newblock NBS Washington, DC.

\bibitem[Kennedy and O'Hagan, 2000]{kennedy2000predicting}
Kennedy, M.~C. and O'Hagan, A. (2000).
\newblock Predicting the output from a complex computer code when fast
  approximations are available.
\newblock {\em Biometrika}, 87(1):1--13.

\bibitem[Krishnamoorthy, 2016]{MR3642449}
Krishnamoorthy, K. (2016).
\newblock {\em Handbook of statistical distributions with applications}.
\newblock CRC Press, Boca Raton, FL, second edition.

\bibitem[Lattimore and Szepesv{\'a}ri, 2020]{lattimore2020bandit}
Lattimore, T. and Szepesv{\'a}ri, C. (2020).
\newblock {\em Bandit algorithms}.
\newblock Cambridge University Press.

\bibitem[Peherstorfer et~al., 2018]{peherstorfer2018survey}
Peherstorfer, B., Willcox, K., and Gunzburger, M. (2018).
\newblock Survey of multifidelity methods in uncertainty propagation,
  inference, and optimization.
\newblock {\em Siam Review}, 60(3):550--591.

\bibitem[Pletcher et~al., 2012]{pletcher2012computational}
Pletcher, R.~H., Tannehill, J.~C., and Anderson, D. (2012).
\newblock {\em Computational fluid mechanics and heat transfer}.
\newblock CRC press.

\bibitem[Reid et~al., 2020]{reid2020probabilistic}
Reid, T.~W., Ipsen, I.~C., Cockayne, J., and Oates, C.~J. (2020).
\newblock A probabilistic numerical extension of the conjugate gradient method.
\newblock {\em arXiv preprint arXiv:2008.03225}.

\bibitem[Saad, 2003]{saad2003iterative}
Saad, Y. (2003).
\newblock {\em Iterative methods for sparse linear systems}.
\newblock Society for Industrial and Applied Mathematics (SIAM), Philadelphia,
  PA, second edition.

\bibitem[Shahriari et~al., 2015]{shahriari2015taking}
Shahriari, B., Swersky, K., Wang, Z., Adams, R.~P., and De~Freitas, N. (2015).
\newblock Taking the human out of the loop: A review of bayesian optimization.
\newblock {\em Proceedings of the IEEE}, 104(1):148--175.

\bibitem[Tong et~al., 1992]{tong1992multilevel}
Tong, C.~H., Chan, T.~F., and Kuo, C.~J. (1992).
\newblock Multilevel filtering preconditioners: Extensions to more general
  elliptic problems.
\newblock {\em SIAM Journal on Scientific and Statistical Computing},
  13(1):227--242.

\bibitem[Trefethen and Bau, 1997]{MR1444820}
Trefethen, L.~N. and Bau, III, D. (1997).
\newblock {\em Numerical linear algebra}.
\newblock Society for Industrial and Applied Mathematics (SIAM), Philadelphia,
  PA.

\end{thebibliography}
